\documentclass[11pt]{amsart}
\usepackage{amsmath, amsthm, amscd, amssymb, amsfonts, amsxtra, verbatim}
\usepackage{pb-diagram}
\usepackage[all,cmtip]{xy}
\usepackage[utf8]{inputenc}
\usepackage{enumerate}
\usepackage{tgpagella}
\usepackage{makecell}
\theoremstyle{plain}
\newtheorem{lema}{Lemma}[section]
\newtheorem{teo}[lema]{Theorem}
\newtheorem{coro}[lema]{Corollary}

\newtheorem{prop}[lema]{Proposition}

\theoremstyle{definition}
\newtheorem{defi}[lema]{Definition}
\newtheorem{ejem}[lema]{Example}

\newtheorem{quest}[lema]{Question}

\newtheorem{rem}[lema]{Remark}

\theoremstyle{remark}

\numberwithin{equation}{section}

\newcommand{\N}{\mathbb N}
\newcommand{\Z}{\mathbb Z}
\newcommand{\Q}{\mathbb Q}
\newcommand{\R}{\mathbb R}

\newcommand{\C}{\mathbb C}

\newcommand{\f}{\frac}
\newcommand{\tf}{\tfrac}

\newcommand{\sk}{\smallskip}
\newcommand{\msk}{\medskip}

\newcommand{\g}{\gamma}
\newcommand{\G}{\Gamma}
\newcommand{\ld}{\lambda}
\newcommand{\Ld}{\Lambda}

\title[The eta function and $\eta$-invariant of $\mathbb{Z}_{2^r}$-manifolds]{The eta function and eta invariant \\ of $\mathbb{Z}_{2^r}$-manifolds}
\address{CIEM, Universidad Nacional de C\'ordoba (UNC), CONICET, FaMAF, C\'ordoba, Rep\'ublica Argentina.}
\email{podesta@famaf.unc.edu.ar}
\author{Ricardo A.\@ Podest\'a}
\thanks{Partially supported by CONICET, FONCyT and SECyT-UNC}

\begin{document}
\bibliographystyle{plain}

\begin{abstract}
We compute the eta function $\eta(s)$ and its corresponding $\eta$-invariant for the Atiyah-Patodi-Singer operator $\mathcal{D}$ acting on an orientable compact flat manifold of dimension $n =4h-1$, $h\ge 1$, and holonomy group $F\simeq \Z_{2^r}$, $r\in \N$. We show that $\eta(s)$ is a simple entire function times $L(s,\chi_4)$, the $L$-function associated to the primitive Dirichlet character modulo 4.
The $\eta$-invariant is 0 or equals $\pm 2^k$  for some $k\ge 0$ depending on $r$ and $n$. 
Furthermore, we construct an infinite family $\mathcal{F}$ of orientable $\Z_{2^r}$-manifolds with $F\subset \mathrm{SO}(n,\Z)$. 
For the manifolds $M\in \mathcal{F}$ we have $\eta(M)=-\tf12|T|$, where $T$ is the torsion subgroup of $H_1(M,\Z)$, and that $\eta(M)$ determines the whole eta function $\eta(s,M)$.
\end{abstract}

\maketitle
\setcounter{tocdepth}{1}

\section{Introduction} \label{sec intro}
\subsubsection*{Eta series and $\eta$-invariant}
Let $M$ be an oriented compact Riemannian manifold of dimension $n=4h-1$, $h \ge 1$, and consider the Atiyah-Patodi-Singer operator $\mathcal{D}$ (\textit{APS-operator} for short) defined on the space of smooth even forms $\Omega^{ev}(M) = \bigoplus_{p=0}^{2h-1} \Omega^{2p}(M)$ by
$$\mathcal{D} \phi = (-1)^{h+p-1} (*d-d*) \phi$$
with $\phi\in \Omega^{2p}(M)$, where $\Omega^{2p}(M)$ denotes the set of degree $2p$ forms. 
This operator is closely related to the signature operator. In fact, 
$\mathcal{D}$ is the tangential boundary operator of the signature operator $\mathcal{S}$ acting on a $4h$-dimensional manifold $\tilde M$ having $M$ as its boundary.

By compactness of $M$, $\mathcal{\mathcal{\mathcal{D}}}$ has a discrete spectrum, $\mathrm{Spec}_\mathcal{D}(M)$, of real eigenvalues $\lambda$ with finite multiplicity $d_\lambda$ which accumulate only at infinity. The \emph{eta series}
\begin{equation} \label{Eta(s)}
\eta(s) = \sum_{0 \not = \lambda \in \mathrm{Spec}_\mathcal{D}(M)} sign(\lambda) \, |\lambda|^{-s}, \qquad \mathrm{Re}(s)> n,
\end{equation}
defines a holomorphic function having a meromorphic continuation to $\C$, also denoted by $\eta(s)$, having (possibly) simple poles in the set 
$\{n-k : k\in \N_0\}$. Remarkably, $\eta(s)$ is holomorphic at $s=0$ and the value $\eta = \eta(0)$ 
is called the \emph{eta invariant} of $\mathcal{D}$.

Both $\mathcal{D}$ and $\eta(s)$ were first introduced and studied by Atiyah, Patodi and Singer in a sequel of 3 classical papers \cite{APS1}, where they also proved the regularity of $\eta(s)$ at the origin in the case of odd dimension. The finiteness of $\eta$ in any dimension is due to P.\@ B.\@ Gilkey 
(\cite{Gi}). Actually, the results in \cite{APS1} and \cite{Gi} are valid for arbitrary elliptic differential operators. 

\msk
Eta series and $\eta$-invariants have been an active area of research since their appearance in \cite{APS1}. A lot of progress have been made mainly by Peter Gilkey, Werner Müller, Xianzhe Dai, Weiping Zhang, Robert Meyerhoff, Mingquing Ouyang and Sebastian Goette 
among others. 
Eta series and $\eta$-invariants have been studied in several contexts. For instance, 
in equivariant settings (Donelly '76, Zhang '90 and Goette '99, '00, '09),
in relation to connective $K$-theory (Gilkey '84 and Barrera-Ya\~nez--Gilkey '99, '03),
manifolds with boundary (Müller '93, '94, Bunke '95, '15 and Dai '02, '06)
and cobordism (Bahri-Gilkey '87, Gilkey '88, '88, '97, Gilkey-Botvinnik '95, '96 and Dai '05),
flat vector bundles (Zhang '04 and Ma-Zhang '06, '06, '08),
even dimensions (Gilkey '85, Dai '12 and Dai-Zhang '15),
determinant lines (Dai-Freed '94, '95), adiabatic limits (Zhang '94 and Dai '06), 
Rokhlin congruences (Zhang '92, '94), relations with $L$-functions and modular forms (Atiyah-Donnelly-Singer '83, '84, Müller '90, Bismut-Cheeger '92 and Han-Zhang '04, '15), etc. 

They were also studied, and in some cases computed, for certain classes of manifolds. Namely, compact flat manifolds 
(\cite{GMP}, \cite{MP.TAMS06}, \cite{MP.PAMQ09}, \cite{MP.JGA11}, \cite{MP.AGAG12}, \cite{Po.UMA05}, \cite{Sz1}, \cite{Sz2}), spherical space forms (\cite{Gi1}, \cite{Gi4}, \cite{Gi5}, \cite{Gi6}, \cite{GiBo1}), hyperbolic manifolds (\cite{LR}, \cite{Me3}, \cite{Me4}, \cite{Ou1}, \cite{Ou2}, \cite{Ou3}) and orbifolds (\cite{Fa}).

\subsubsection*{Compact flat manifolds}
Any orientable compact flat manifold (in what follows \textit{cfm} for short) is isometric to 
$M_\G = \G \backslash \R^n$, with $\G$ an orientable Bieberbach group, i.e.\@ a discrete, cocompact, torsion-free subgroup of 
the orientation preserving isometry group $\mathrm{I}^+(\R^n) = \mathrm{SO}(n) \ltimes \R^n$ of $\R^n$. 
Thus, 
$\G = \langle \gamma= BL_b, L_\Lambda \rangle$, with $L_\Lambda = \{L_\lambda : \lambda \in \Lambda \}$,
where $B \in \mathrm{SO}(n)$, $L_b$ denotes translation by $b\in \R^n$, $\g^k \ne \mathrm{Id}$ for every $k\in \N$ and $\Lambda$ is a $B$-stable lattice in $\R^n$. We will usually identify the point group with the holonomy group, that is  $F\simeq \Lambda \backslash \G$.

Since $BL_\lambda B^{-1}=L_{B\lambda} \in L_\Lambda$ for any $B\in \mathrm{SO}(n), \lambda \in \Lambda$, conjugation by $F$ in 
$\Lambda \simeq \Z^n$ defines the integral holonomy representation $\rho$ of $\G$ (which does not determine $M_\Gamma$ uniquely, in general). A \textit{$G$-manifold} is a cfm with holonomy group $F\simeq G$. In this paper we will be concerned with the case $F\simeq \Z_{2^r}$, for $r\in \N$.

\msk 
In \cite{MP.JGA11}, we give a general expression for $\eta(s)$ and $\eta$ for $\mathcal{D}$ acting on an arbitrary orientable cfm 
$M_\G$ (see Theorems 3.3, 3.5 and 4.2). Also, simpler expressions can be found in its sequel \cite{MP.AGAG12} in the case of cyclic holonomy group (see Proposition \ref{AGAG}). There, it is shown that the computation of $\eta(s)$ can be reduced to the case when $F$ is cyclic (\cite{MP.AGAG12}, 
Proposition~5.1). 
This, enabled us to compute the $\eta$-invariant of some nice families of cfm's, such as $F$-manifolds with 
$F\simeq \Z_p, \Z_p\times \Z_q, \Z_2^k$ and $\Z_p^k$, with $p,q$ odd primes and $k\ge 2$; or even with non-abelian holonomy group $F$, where $F$ is of order 8 or $F$ is metacyclic (dihedral and of odd order). Expressions for $\eta(s)$ and $\eta$ on cfm's for the spin Dirac operator $D$ were studied in \cite{MP.TAMS06} (the general case and $\Z_2^k$-manifolds), in \cite{Po.UMA05} ($\Z_4$-manifolds) and in \cite{MP.PAMQ09} and \cite{GMP} ($\Z_p$-manifolds, $p$ odd prime).

\subsubsection*{Motivation and results}
As we have mentioned, in \cite[\S 5]{MP.AGAG12}, a method to reduce the computation of the $\eta$ function for a manifold with arbitrary holonomy group $F$ to the case of cyclic holonomy group is presented. 
Hence, for a general cyclic holonomy group 
$F\simeq \Z_N = \Z_{p_1^{r_1}} \times \cdots \times \Z_{p_t^{r_t}}$, with $p_1,\ldots,p_t$ different primes, 
and thus one is basically led to the study of the contributions of each $\Z_{p_i^{r_i}}$ to the computation of $\eta(\Gamma)$. 
As a step in this direction, it is the goal of this paper to cover the case of the even prime $2$. We are thus interested in the computation of $\eta(s)$ and $\eta$ for the operator $\mathcal{D}$ on arbitrary orientable $\Z_{2^r}$-manifolds 
of odd dimension $n=2m+1$, $m$ odd. 
In this way, the present paper can be considered as a natural continuation of \cite{MP.JGA11} and \cite{MP.AGAG12}. We remark here that $\Z_{2^r}$-manifolds are not classified (as it is the case for $\Z_p$ manifolds, $p$ prime).

\sk
A brief outline of the paper is as follows. In Section \ref{sec rotation}, we study the rotation angles of a matrix $B\in \mathrm{GL}(n)$ of order $2^r$ and its characteristic polynomial. Some trigonometric identities involving sines and cotangents at special angles of the form $\tf{k\ell \pi}{2^r}$, with $k,\ell$ odd integers, play a key role in the computation of $\eta(s)$ and $\eta$, allowing great simplifications in the steps of the proofs, which lead to the final expressions in the statements. For clarity, they are presented at the end in the Appendix.

Sections \ref{sec etas} and 4 are devoted to the computation of the $\eta$-function and $\eta$-invariant, respectively, for arbitrary 
$\Z_{2^r}$-manifolds in all dimensions $n=2m+1$, $m$ odd. Let $M_\G$ be a $\Z_{2^r}$-manifold with $\gamma=BL_b$ the generator of $\Lambda\backslash \G$. 
In Theorem~\ref{teo.etasz2r} we show that, in the non-trivial cases,  
$$\eta(s) = \tf{2 \eta}{(2^{r-1-\nu}\pi \lambda_{_B})^s} L(s,\chi_4),$$
where $\nu \in [0,r-2]\in \Z$ and $\lambda_{_B}\in \R$ are certain constants depending on the metric,
and $L(s,\chi_4)$ is the $L$-function associated to the primitive Dirichlet character modulo $4$. 
Here, the $\eta$-invariant 
has the simple expression (see Proposition \ref{prop eta0})
$$\eta = \pm  \, 2^{f(B)-2},$$ 
where $\pm$ is a sign depending subtly on $\g$ and $f(B)$ is the number of irreducible factors of the characteristic polynomial of $B$ on $\Z[x]$ (see \eqref{fb}). This gives an easy and direct way of computing the invariant.
We give some examples to illustrate the method.

Next, in Section \ref{sec.family}, we introduce an infinite family $\mathcal{F}$ of orientable $\Z_{2^r}$-manifolds of 
$\dim n=4h-1$, $h\ge 1$, each having integral holonomy representation. 
In Proposition \ref{teoH1} we compute the first (co)homology groups over $\Z$ and $\Z_2$ of $M\in \mathcal{F}$. This allows us to get the $\eta$-invariant of $M$ (and hence $\eta(s)$, by \eqref{generaletas} and \eqref{generaleta0}) in topological terms; namely 
$$\eta(M) = -\tf 12 |\mathrm{Tor}(H_1(M,\Z))|.$$
We also posed some queries, see Questions \ref{q1}, \ref{q2} and \ref{q3}.

In Section \ref{sec eta*}, we show that for any $\Z_{2^r}$-manifold $M$, there is some $M_{k(r)} \in \mathcal{F}$ such that
$\eta(M) = \eta(M_{k(r)})$. This allows us to prove that the set of possible values of $\eta(M)$, 
with $M$ ranging over all $\Z_{2^r}$-manifolds, $r\ge 1$, is $0$ or a non-negative power of $2$.
Then, we show that there are infinite families of $\Z_{2^r}$-manifolds with $\eta=2^k$, for each $k$ (with growing dimensions). 
Moreover, there is a number $n_{r,k}$ such that for every $n\ge n_{r,k}$ there is a $\Z_{2^r}$-manifold of dimension $n$ with $\eta=2^k$. 
As a result, varying $r$ in the natural numbers we have
$$\eta(\text{$\Z_{2^r}$-manifolds}) = \{0\}\cup \{\pm 2^{k} : k\in \N_0\}.$$

Finally, in the last section, we compare expression \eqref{cyclic eta0} for the $\eta$-invariant of any $\Z_{2^r}$-manifold, with Donnelly's expression \eqref{doneli}, valid only for those $\Z_{2^r}$-manifolds having holonomy group $F\subset\mathrm{SO}(n,\Z)$ (see Proposition \ref{etaz2rdoneli}). In this case, we have $\eta=\pm 2^{f(B)-2}=\pm 2^{c(B)-2}$, where $c(B)$ is the number of orbits of the action of $B$ on the canonical basis vectors, thus giving an alternative way of computing the invariant.

\section{Rotation angles for order $2^r$ matrices} \label{sec rotation}
As we shall later see, the results in this section are crucial for the determination of $\eta(s)$ in Theorem \ref{teo.etasz2r}. 
Here, we will study the rotation angles of a matrix in $\mathrm{GL}(n,\Z)$ of order $2^r$.
In the Appendix (for clarity), we will compute some trigonometric identities related to them.

\sk 
For any $d\in \N$, let $\mathcal{U}_d =\{\omega \in \C : \omega^d=1\}$ be the group of complex $d$-th roots of unity and denote by
$\mathcal{U}_d^*=\{\omega \in\mathcal{U}_d : \mathrm{ord}(\omega) =d\}$ the subgroup of primitive roots in $\mathcal{U}_d$.
Then, for any $N\in \N$, we clearly have $\mathcal{U}_N = \bigcup_{d \mid N} \mathcal{U}_d^*$.
The cyclotomic polynomial of order $d$ is defined by $\Phi_d(x)= \prod_{\omega \in \mathcal{U}_d^*} (x-\omega) \in \C[x]$. 
It is known that $\Phi_d(x) \in \Z[x]$ is irreducible over $\Q$ of degree $\varphi(d)$, where $\varphi$ is the Euler totient function. We have the relation 
\begin{equation}\label{cyclotomic decomp}
x^n-1 = \prod_{d\mid n} \Phi_d(x).
\end{equation}

\begin{lema}\label{lema charpol}
Let $B \in \mathrm{GL}(n,\Z)$ of order $N \in \N$. Then, the characteristic polynomial $p_{_B}(x)$ of $B$ has the prime factorization
\begin{equation}\label{eq. charpol}
p_{_B}(x) = \prod_{d\mid N} \Phi_d(x)^{c_d}
\end{equation}
in $\Z[x]$, with $c_d\ge0$ for any $d\mid N$ and $c_N\ge 1$. Also, $\deg p_{_B} = \sum_{d\mid N} c_d \cdot \varphi(d)$.
\end{lema}
\begin{proof}
We know that $p_{_B}(x)$ is a monic polynomial in $\Z[x]$, all of whose roots are in $\mathcal{U}_N$.
Since $\Z[x]$ is a unique factorization domain, $p_{_B}(x)=\prod_j p_j(x)$, with $p_j(x) \in \Z[x]$ monic irreducible for each $j$.
We will show that the only monic irreducible polynomial in $\Z[x]$, with roots in $\mathcal{U}_h$ is $\Phi_h(x)$, $h\mid N$. Let $\alpha$ be a root of $p_{_B}(x)$ and let $\sigma \in \mathrm{Gal}(\Q(\alpha)/\Q)$
with $\sigma(\alpha)=\alpha^k$ for $(k,h)=1$. We have $p_{_B}(\alpha^k) = p_{_B}(\sigma(\alpha)) = \sigma(p_{_B}(\alpha)) = 0$, 
for any $(k,h)=1$. Thus, $\Phi_h(x) \mid p_{_B}(x)$ in $\Q[x]$, hence in $\Z[x]$, and, by irreducibility, it must be one of the $p_j(x)$'s. Thus, we get \eqref{eq. charpol} with $c_d\ge 0$ for each $d\mid N$. 
Clearly, we must have $c_N\ge 1$ for $B$ to have order $N$.
The assertion on the degree is obvious. 
\end{proof}

For $B\in \mathrm{GL}(n,\Z)$ we define $f(B)$ to be the \emph{number of irreducible factors} in the prime decomposition of the characteristic polynomial $p_{_B}(x)$ of $B$ in $\Z[x]$. Thus, by \eqref{eq. charpol}, 
\begin{equation}\label{fb}
1\le f(B)= \sum_{d \, \mid \, o(B)} c_d
\end{equation}
where $o(B)$ is the order of $B$. If $B= (\begin{smallmatrix} B' & \\ & 1 \end{smallmatrix})$, then $f(B') = f(B)-1$ is the number of irreducible factors in the prime decomposition of $p_{_{B'}}(x)= p_{_B}(x)/(x-1)$.

\sk 
Let $n=2m+1$.
Since $\mathrm{SO}(n)$ is a compact connected Lie group, it contains a maximal torus 
$T_n = \{x(t_1,\ldots,t_m) : t_1,\ldots,t_m\in \R \}$, where
$$x(t_1,\ldots,t_m)  := \mathrm{diag}(x(t_1),\ldots,x(t_m),1)$$
with $x(t) = (\begin{smallmatrix} \cos t & -\sin t \\ \sin t & \cos t \end{smallmatrix})$, $t\in \R$.
Also, $T_{n-1} = \{x(t_1,\ldots,t_m) : t_1,\ldots,t_m\in \R \}$, where now
$x(t_1,\ldots,t_m)  = \mathrm{diag}(x(t_1),\ldots,x(t_m))$.

\sk 
Let $\G\subset \mathrm{I}^+(\R^n)$ be a Bieberbach group and for any $BL_b\in \G$ put 
$$n_B := \dim (\R^n)^B.$$ 
Note that $n_{\pm B}$ is the multiplicity of the $(\pm 1)$-eigenvalues of $B$.
Since $B\in \mathrm{SO}(n)$, $B$ is conjugate to some element $x_B\in T_n$. Thus, there is $C\in \mathrm{GL}(n)$ 
such that $CBC^{-1}=x_B$. Since $x_B$ fixes $e_n$, $B$ fixes $u=C^{-1}e_n$ and hence $n_{x_B}=n_B\ge 1$ 
(this is known by other methods, see for instance \cite{MR}).
Also, note that $p_{_B}(x)=p_{x_B}(x)$ and therefore $f(B)=f(x_B)$.

\begin{prop}\label{lema angles}
Suppose $\G$ is an orientable Bieberbach group of dimension $n=2m+1$.
Let $BL_b \in \G$ with $B$ of order $N$. Then, $B$ is conjugate in $T_{2m}$ to $x_B=x(t_1,t_2\ldots,t_m)$ or to 
$x_B'=x(-t_1,t_2,\ldots,t_m)$ with
$$x_B = x(\underbrace{ j_{s,1} \tf{2\pi}{d_s}, \ldots, j_{s,f_s} \tf{2\pi}{d_s} }_{c_{d_s}\ge 1},\ldots, 
\underbrace{ j_{1,1} \tf{2\pi}{d_1}, \ldots, j_{1,f_1} \tf{2\pi}{d_1} }_{c_{d_2}}, \underbrace{0}_{c_1-1}),$$
where
$1 = d_1 < d_2 < \cdots < d_s = N$ are all the divisors of $N$ and each $c_{d_j}$ is the exponent of 
$\Phi_{d_j}(x)$ in the prime factorization of $p_{x_B}(x)$ in \eqref{eq. charpol}, such that
$$(j_{i,k}, d_i)=1, \qquad 1\le j_{i,k} \le \lfloor \tf{d_i}{2} \rfloor, \qquad  1 \le k \le f_i = \tf{\varphi(d_i)}{2},$$
for each $2\le i \le s$.
\end{prop}
\begin{proof}
Since $B' \in \mathrm{SO}(n-1)$, $B'$ is conjugate to some element $x_B$ in $T_{n-1}$. 
There are two conjugacy classes in $\mathrm{SO}(n-1)$, and hence $B$ is conjugate to $x_B=x(t_1,t_2,\ldots,t_m)$ or to $x_B'=x(-t_1,t_2,\ldots,t_m)$ with $0\le t_i < 2\pi$ for $i=1,\ldots,m$ (see the comments in between (3.8) -- (3.10) in \cite{MP.JGA11}).

Since $B$ is of order $N>1$, it is clear that the rotation angles must be of the form $\tf{2j\pi}{N}$, for certain 
$1\le j \le [\tf{N-1}2]$. Also, $o(B)=N$ implies $c_N \ge 1$. 
Thus, if $\omega \in \mathcal{U}_N^*$, all the $\varphi(N)$ primitive $N$-th roots of unity $\{\omega^j : (j,N)=1\}$ are roots of 
$p_{_B}(x)$. Since complex roots appear in conjugate pairs, it suffices to consider the angles in $[0,\pi]$. Hence, the rotation angles corresponding to $\mathcal{U}_N^*$ are the $\tf{\varphi(N)}2$ angles $\tf{2j\pi}{N}$ with $(j,N)=1$, $1\le j\le [\tf N2]$, each with multiplicity $c_N$. Similar arguments apply for every $\zeta \in \mathcal{U}_d^*$ with $d\mid N$ and $c_d \ge 1$. 
\end{proof}

As a result, in the case $N=2^r$ we can be more precise.
\begin{coro}\label{coroz2r}
Let $M_\G$ be a $\Z_{2^r}$-manifold of dimension $n$, with $F = \langle B \rangle$. Then, we have
$$p_{_B}(x) = (x-1)^{c_1} \prod_{k=1}^r (x^{2^{k-1}}+1)^{c_{2^k}}$$
with $n=\deg p_{_B} = c_1+ \sum_{k=1}^r  2^{k-1} c_{2^k}$ and $c_1, c_{2^r}\ge 1$.
Also,
\begin{equation}\label{xB2r}
x_B = x(\underbrace{ \tf{\pi}{2^{r-1}}, \tf{3\pi}{2^{r-1}}, \ldots,  \tf{(2^{r-1}-1)\pi}{2^{r-1}} }_{c_{2^r}}, \ldots,
\underbrace{ \tf{\pi}{4}, \tf{3\pi}{4} }_{c_8}, \underbrace{ \tf{\pi}{2} }_{c_4}, \underbrace{ \pi}_{c_2}, 
\underbrace{ 0 }_{c_1-1}).
\end{equation}
\end{coro}
\begin{proof}
First we note that we can assume that $B\in \mathrm{O}(n,\Z)$. In fact, $B \in \mathrm{O}(n)$ and since $\Lambda\simeq \Z^n$ we have $\rho(B) \in \mathrm{GL}(\Lambda) \simeq \mathrm{GL}(\Z^n) \simeq \mathrm{GL}(n,\Z)$ and hence $\det(\rho(B))=\det(B)$ and $p_B(x) = p_{\rho(B)}(x)$.
By \eqref{cyclotomic decomp} and  $x^{2^k}-1=(x^{2^{k-1}}-1)(x^{2^{k-1}}+1)$ we have that 
$\Phi_{2^k}(x) = x^{2^{k-1}}+1$, $k\ge 1$. Since the divisors of $2^r$ are $1,2,2^2,\ldots, 2^r$, by Lemma \ref{lema charpol} with $N=2^r$ we have that $p_{_B}(x)$ equals $(x-1)^{c_1} (x+1)^{c_2} (x^2+1)^{c_4} (x^4+1)^{c_8} \cdots (x^{2^{r-1}}+1)^{c_{2^r}}$
with $c_{2^r}\ge 1$, as we wanted. 

Relative to the angles, all the $j_{i,k}$'s in Proposition \ref{lema angles} are odd, and hence $B$ is conjugate in $T_{2m}$ to 
$x_B = x(t_1,t_2,\ldots,t_m)$ or to $x_B' = x(-t_1,t_2,\ldots,t_m)$ with $x_B$ as in \eqref{xB2r}.
Clearly, we have $n= 2^{r-1} c_{2^r} + \cdots + 4c_8 + 2c_4 + c_2 +(c_1-1) + 1 = \deg p_{_B}$. Since $n_B$ is the multiplicity of 1 as eigenvalue of $B$, and $n_B\ge 1$, we have $c_1\ge 1$.
\end{proof}

\section{The eta function} \label{sec etas}
Let $\G$ be an $n$-dimensional Bieberbach group with translation lattice $\Lambda$ and point group $F$ and  
let $\g=BL_b\in \G$. If $x_B = x(t_1,\ldots,t_m) \in T_{n-1}$
is conjugate to $B$, denote the angles $t_1,\ldots,t_m$ of $B$ by $t_1(x_B), \ldots,t_m(x_B)$ and put 
$$(\Lambda \backslash \G)' =\{ BL_b\in \Lambda \backslash \G : B\in F_1'\}$$ 
where $F_1' = \{B \in F_1 : t_i(x_B) \not \in \pi\Z, \: 1\le i \le m \}$ with $F_1 = \{B \in F : n_B=1\}$. 
If $B\in F_1$, choose $v_B \in \Lambda^*$ such that
\begin{equation}\label{vB}
(\Lambda^*)^B = \Z v_B \qquad \text{and put} \qquad \lambda_{_B}=\|v_B\|. 
\end{equation}
Then, if $o(B)$ denotes the order of $B$, we have (see \cite{MP.JGA11}, Lemma 4.1)
\begin{equation}\label{elegama}
\ell_\gamma := o(B) \langle v_B, b \rangle \in \Z.
\end{equation}

\sk 
We now recall the expressions of $\eta(s)$ for $\Z_N$-manifolds, that will be our starting point. 
For $\alpha \in  (0,1]$ let $\zeta(s,\alpha)=\sum_{n>0} (n+\alpha)^{-1}$ be the Riemann-Hurwitz zeta function for $\mathrm{Re}(s)>1$.

\begin{prop}[\cite{MP.JGA11}, Proposition 2.2] \label{AGAG}
Let $N=2^r$ and assume that $M_\Gamma$ be an orientable $\Z_N$-manifold of dimension $n = 2m+1$, with $m$ odd, and $\Gamma = \langle \gamma , L_\Lambda \rangle$ where $\gamma = BL_b$.
If $r=1$ or $n_B > 1$ then $\eta(s)=0$. On the contrary, if $r\ge 2$ and $n_B =1$ then we have
\begin{equation}  \label{etas cyclic}
\eta(s) = -\sigma_{v_B} \, \tf{2^{m+1}}{N(2 \pi \lambda_{_B})^s}  \sum_{\substack{k=1 \sk \\ \g^k \in (\Ld \backslash \G)' }}^{N-1} \, \tf{1}{o(B^k)^s} \, \Big( \prod_{j=1}^m \sin (k t_j) \Big) \, \sum_{j=1}^{o(B^k)-1} \sin \big( \tfrac{2\pi j k \ell_\gamma}{N} \big) \,  \zeta(s, \tfrac{j}{o(B^k)})
\end{equation}
where $v_B$ and $\ell_\g$ are as in \eqref{vB} and \eqref{elegama} respectively, and 
\begin{equation}  \label{cyclic eta0}
\eta = -\sigma_{v_B}\, \tf{2^{m}}{N} \sum_{k=1}^{N-1} \sum_{\g^k \in (\Ld \backslash \G)'} 
\big( \prod_{j=1}^m \sin (k t_j) \big) \cot(\tf{\pi k \ell_\g}{N}), 
\end{equation}
where $\sigma_{v_B} \in \{\pm 1\}$ is a sign depending on the conjugacy class of $x_B$ in $T_{n-1}$.
\end{prop}
\noindent
\textit{Note.} See \cite{MP.JGA11}, (3.10)--(3.11), for details on the sign $\sigma_{v_B}$.

\sk
We now give the eta function of $\mathcal{D}$ for an arbitrary $\Z_{2^r}$-manifold $M_\G$.
We recall that the class of ${\Z_{2^r}}$-manifolds are not classified; as it is indeed the case for $\Z_p$-manifolds with $p$ an odd prime (\cite{Ch}) or for cfm's in low dimensions (see \cite{HW} for  $\dim =3$, \cite{BB+} for $\dim=4$ and \cite{CS} for $\dim =5,6$). 
However, all the information needed to compute $\eta(s)$ is (roughly) contained in the angles of the rotational part of the elements 
of $\G$. 

\sk 
We will express $\eta(s)$ in terms of a Dirichlet $L$-function. We recall that for $\chi : \Z \rightarrow \C^*$ a Dirichlet character modulo $N$ one has the  Dirichlet series $L(s,\chi) = \sum_{n=1}^\infty \f{\chi(n)}{n^s}$, absolutely convergent for $\mathrm{Re}(s) >1$. It has an analytic continuation to the whole $\C$ given by 
\begin{equation}\label{L&zetas}
L(s,\chi) = \tf{1}{N^s}\sum_{j=1}^N \chi(j) \, \zeta(s, \tf jN) \,.
\end{equation}
We will be mainly concerned with $\chi_4$, the \textit{primitive} character mod $4$, defined by 
$\chi_4(1) = 1$, $\chi_4(3) = -1$ and $\chi_4(0) = \chi_4(2) = 0$, that is 
$\chi_4(n) = \sin(\tf{\pi n}{2})$ for $n\in \Z$.

\msk
The promised result is as follows.
\begin{teo} \label{teo.etasz2r}
Let $M_\G$ be any orientable $n$-dimensional $\Z_{2^r}$-manifold with $r\in \N$, $n \equiv 3$ mod 4 and $\G= \langle \g= BL_b, L_\Lambda \rangle$. 
For $\ell_\g$ as in \eqref{elegama} put $\ell_\g = 2^\nu \ell$ with $\ell$ odd and $\nu \ge 0$. If $n_B>1$ or $r=1$, then $\eta(s)=0$. If $r\ge 2$ and $n_B=1$ then Spec$_\mathcal{D}(M_\G)$ is asymmetric and we have 
\begin{equation}\label{generaletas}
\eta(s) = \tf{\sigma_\g\, 2^{f(B)-1}}{(2^{r-1-\nu} \pi  \lambda_{_B})^{s}} \, L(s,\chi_4)
\end{equation}
with $0\le \nu \le r-2$ and $\sigma_\g = \sigma_{v_B} \, (-1)^{[\f{\ell}{2}]+1} \in \{\pm 1\}$, where $\chi_4$ is the primitive Dirichlet character modulo 4 while $f(B)$ and $\lambda_B$ are as in \eqref{fb} and \eqref{vB} respectively. 
\end{teo}
\begin{proof}
We will apply Proposition \ref{AGAG}. It is known that $\Z_2$-manifolds have symmetric spectrum and hence trivial $\eta(s)$ (\cite{MP.JGA11}, Proposition 3.11, or \S 6.1.1 in \cite{MP.AGAG12}). Also, $\eta(s)=0$ if $n_B>1$. Thus, from now on, we assume that  
$r\ge 2$ and $n_B=1$. Let $n=2m+1$, $m$ odd. 
We now study the ingredients in \eqref{etas cyclic}. The angles are given by \eqref{xB2r}.
Also, $n_B=1$ implies $c_1=1$ and $\g^k \in (\Lambda \backslash \G)'$ if and only if $k$ is odd and $c_2=0$. Thus,
\begin{equation}\label{xB2r0}
x_B = x(\underbrace{\tf{\pi}{2^{r-1}}, \tf{3\pi}{2^{r-1}}, \ldots,  \tf{(2^{r-1}-1)\pi}{2^{r-1}} }_{c_{2^r}}, \ldots, 
\underbrace{\tf{\pi}{8}, \tf{3\pi}{8},  \tf{5\pi}{8}, \tf{7\pi}{8}}_{c_{16}}, \underbrace{\tf{\pi}{4}, \tf{3\pi}{4}}_{c_8}, 
\underbrace{\tf{\pi}{2}}_{c_4}).
\end{equation}
Note that $o(B^k)=2^r$ for every $k$ odd. Relative to the dimension, we have that
\begin{equation}\label{mB}
m = 2^{r-2} c_{2^r} + \cdots + 4 c_{16} + 2 c_8 + c_4.
\end{equation}
By Corollary \ref{coroz2r}, we know that $c_{2^r} \ge 1$. The condition $m$ odd implies that $c_4$ is also odd, and, in particular, $c_4 \ge 1$.

By putting all this information in \eqref{etas cyclic} we get 
\begin{equation} \label{etaux1}
\eta(s) = - \tf{\sigma_{v_B} \, 2^{m+1-r}}{(2^{r+1}\pi \|v_B\|)^s} \, \sum_{k\in I_{2^r}^*}
\, P_k(x_B) \,  \sum_{j=1}^{2^r-1} \sin \big( \tf{\pi jk\ell_\g}{2^{r-1}} \big) \,  \zeta(s,\tf{j}{2^r}) \,,
\end{equation}
where we have used the notations $I_{N}^*=\{1\le i \le N : i \text{ odd}\}$ and $P_k(x_B) = \prod_{j=1}^m \sin(kt_j)$. 
By \eqref{xB2r}, we have
$$ P_k(x_B) = \Bigg( \prod_{\substack{j=1 \sk \\ j \text{ odd}}}^{2^{r-1}} 
\sin (\tf{jk\pi}{2^{r-1}}) \Bigg)^{c_{2^r}}
\Bigg( \prod_{\substack{j=1 \sk \\ j \text{ odd}}}^{2^{r-2}} 
\sin (\tf{jk\pi}{2^{r-2}}) \Bigg)^{c_{2^{r-1}}} \cdots \Bigg( \prod_{\substack{j=1 \sk \\ j \text{ odd}}}^{2^{2}} 
\sin (\tf{jk\pi}{2^{2}}) \Bigg)^{c_8} \big( \sin (\tf{k\pi}{2})\big)^{c_4}.$$
For $r=2$, we have
$P_k(x_B) = (\sin (\tf{k\pi}{2}))^{c_4} = (-1)^{[\tf k2]}$,  since $c_4$ is odd. 
For $r\ge 3$, by Proposition~\ref{prop. sines2} in the Appendix, we get  (this is a first key step)
$$P_k(x_B)  =  (-1)^{[\tf k2]} \, \prod_{i=2}^{r-1} \Bigg( \prod_{j\in I_{2^i}^*}
\sin (\tf{jk\pi}{2^{i}}) \Bigg)^{c_{2^{i+1}}} = 
(-1)^{[\tf k2]} \, \prod_{i=2}^{r-1} \big( \tf{1}{2^{2^{i-1}-1}} \big)^{c_{2^{i+1}}} = (-1)^{[\tf k2]} \, 2^{-S_r},$$
where we have put $S_r := \sum\limits_{i=2}^{r-1} (2^{i-1}-1) \, c_{2^{i+1}}$. 
By \eqref{mB} and  \eqref{fb}, since $c_1=1$, 
$c_2=0$, we have
$$S_r = \sum_{i=2}^{r-1} 2^{i-1} c_{2^{i+1}} - \sum_{i=2}^{r-1} c_{2^{i+1}} = (m-c_4) - (f(B')-c_4) = m-f(B)+1.$$
Therefore, we get
$P_k(x_B) = (-1)^{[\tf k2]} \, 2^{-m + f(B) -1}$ and putting this in \eqref{etaux1} we obtain
\begin{equation}
\label{etaux2}
\eta(s) = -  \tf{\sigma_{v_B} \, 2^{f(B)-r}}{(2^{r+1}\pi \|v_B\|)^s} \, \sum_{k\in I_{2^r}^*}
(-1)^{[\tf k2]}  \underbrace{\sum_{j=1}^{2^r-1} \sin \big( \tf{\pi jk\ell_\g}{2^{r-1}} \big) \, \zeta(s,\tf{j}{2^r})}_{:= \xi_{k,\ell_\g,r}(s)}.
\end{equation}
Now, since $\sin(\tf{(2^r-j)k\pi}{2^{r-1}}) = -\sin(\tf{jk\pi}{2^{r-1}})$ for every $j,k \in \N$, we can write
\begin{eqnarray*}
\xi_{k,\ell_\g,r}(s) &=& 
\sum_{j=1}^{2^{r-1}-1} \sin \big( \tf{\pi jk\ell_\g}{2^{r-1}} \big) \, \big\{ \zeta(s,\tf{j}{2^r}) - \zeta(s,1-\tf{j}{2^r}) \}  \\
& = & \sum_{t=2}^r \sum_{j\in I_{2^{t-1}}^*}
\sin \big( \tf{\pi jk\ell_\g}{2^{t-1}} \big) \, \big\{ \zeta(s,\tf{j}{2^t}) - \zeta(s,1-\tf{j}{2^t}) \},
\end{eqnarray*}
where we have put together the contribution of the angles of the $N$-th roots of unity, $N=2^r$, of the same order $2^t$, $1\le t \le r-1$.

In this way, the sum in \eqref{etaux2} becomes
\begin{eqnarray*}
\sum_{k\in I_{2^r}^*} 
(-1)^{[\tf k2]} \, \xi_{k,\ell_\g,r}(s) &=& \sum_{t=2}^r  \sum_{j\in I_{2^{t-1}}^*} 
\Bigg(  \sum_{k\in I_{2^r}^*} (-1)^{[\tf k2]} \, \sin \big( \tf{\pi jk\ell_\g}{2^{t-1}} \big) \Bigg) \, 
\big\{ \zeta(s,\tf{j}{2^t}) - \zeta(s,1-\tf{j}{2^t}) \big \}.
\end{eqnarray*}
By Proposition \ref{prop. sumsines} in the Appendix, the sum between parentheses equals 
$$\mathcal{S}_{r,t-1}(2^\nu j\ell) = \delta_{t-1,\nu+1}(-1)^{[\f{j\ell}{2}]} 2^{r-1}$$ 
for $t\le r$ (where $\ell_\gamma=2^\nu \ell$ with $\ell$ odd); i.e.\@ it does not vanish only for $t=\nu+2$ and hence $0\le \nu \le r-2$ 
(this is a second key step). 

Thus, since $(-1)^{[\tf{j\ell}{2}]} = (-1)^{[\tf{j}{2}]} (-1)^{[\tf{\ell}{2}]}$, we get 
$$\sum_{k\in I_{2^r}^*} (-1)^{[\tf k2]} \, \xi_{k,\ell_\g,r}(s) =  (-1)^{[\tf{\ell}2]} \, 2^{r-1} \, \sum_{j\in I_{2^{\nu+1}}^*} (-1)^{[\tf j2]}  \, \big\{ \zeta(s,\tf{j}{2^{\nu+2}}) - \zeta(s,1-\tf{j}{2^{\nu+2}}) \big\}.$$
Putting this information in \eqref{etaux2}, we get the expression  
\begin{equation} \label{preetas}
\eta(s) =  \sigma_\g \, \f{ 2^{f(B)-1}}{(2^{r+1}\pi \lambda_{_B})^{s}}  \, \sum_{j \in  I_{ 2^{\nu+1}}^*} (-1)^{[\tf j2]} 
\big( \zeta(s, \tf{j}{2^{\nu+2}}) - \zeta(s,1-\tf{j}{2^{\nu+2}}) \big)
\end{equation}
where $\sigma_\gamma =-\sigma_{v_B} (-1)^{\f{\ell}{2}}$.

\msk 
The last step will be to simplify the above expression. So, denote by $F_\nu(s)$ the function given by the sum in \eqref{preetas}, i.e.\@ 
$$F_\nu(s) = \sum_{j \in  I_{ 2^{\nu+1}}^*} (-1)^{[\tf j2]} 
\big( \zeta(s, \tf{j}{2^{\nu+2}}) - \zeta(s,1-\tf{j}{2^{\nu+2}}) \big).$$
Note that $L(s,\chi_4)=\tf{1}{4^s} \big( \zeta(s,\tf14) - \zeta(s,\tf34) \big)$ by \eqref{L&zetas}.
Thus, $F_0(s)= 4^s L(s,\chi_4)$ and hence \eqref{generaletas} holds for $\nu=0$ (i.e.\@ $t=2$).

We now consider the remaining cases, that is $\nu>0$ (i.e.\@ $t\ge 3$). Suppose $\chi$ is a Dirichlet character mod $k$. Then, there is some primitive Dirichlet character $\psi$ mod $d$, with $d \mid k$ (the conductor), such that $\chi=\psi \cdot  \chi_{1,k}$, where 
$\chi_{1,k}$ is the principal Dirichlet character mod $k$, and 
\begin{equation}\label{L-ind}
L(s,\chi) = L(s,\psi) \prod_{p\mid k} \Big( 1 - \tf{\chi(p)}{p^s} \Big).
\end{equation}
Now, for $t \ge 3$, let $\chi_{2^t}$ be the real Dirichlet character mod $k=2^t$ induced by the primitive character $\chi_4$, that is 
$\chi_{2^t}=\chi_4 \cdot \chi_{1,2^t}$. Thus, $\chi_{2^t}(2j)=0$ and $\chi_{2^t}(2j+1)=(-1)^{[\f j2]}$.

On one hand, since $k=2^t$, by \eqref{L-ind} we get
\begin{equation}\label{L1}
L(s,\chi_{2^t}) = L(s,\chi_4) (1 - \tf{\chi_4(2)}{2^s}) = L(s,\chi_4).
\end{equation}
On the other hand, by \eqref{L&zetas} we have $L(s,\chi_{2^t}) = \tf{1}{2^{ts}} \sum_{j\in I_{2^{t}}^* } \chi(j) \zeta(s,\tf{j}{2^t})$. Thus, since $\chi_{2^t}(2^t-j) = \chi_{2^t}(-1)\chi_{2^t}(j)=-\chi_{2^t}(j)$, we obtain
\begin{equation} \label{L2} 
\begin{split}
L(s,\chi_{2^t}) = \tf{1}{(2^t)^s} \sum_{j\in I_{2^{t-1}}^* } (-1)^{[\f j2]} \, \big( \zeta(s,\tf{j}{2^t}) - \zeta(s,1-\tf{j}{2^t}) \big) 
= \tf{1}{(2^t)^s}  F_{t-2}(s).
\end{split}
\end{equation}
From \eqref{L1} and \eqref{L2}, we have 
$$F_\nu(s) = 2^{\nu s} L(s,\chi_{2^{\nu+2}}) = 2^{t s} L(s,\chi_4),$$ 
with $t=\nu+2\le r$. Putting this in \eqref{preetas} we get \eqref{generaletas} for any $0\le \nu \le r-2$.

\sk
Finally, note that by \eqref{Eta(s)} we have 
$\eta(s) = \sum_{\lambda>0} (d_\lambda^+ - d_{\lambda}^-) \lambda^{-s}$, 
where $d_{\ld}^{\pm}$ is the multiplicity of $\pm \ld$. Therefore, symmetry in the spectrum implies $\eta(s) = 0$ for every $s$, a contradiction. Thus the spectrum is asymmetric and the result follows.
\end{proof}

Notice that, by \eqref{etas cyclic}, the dependence of $\eta(s)$ on the metric is given only by 
$\nu \in \N_0$ (with $\ell_\g = 
2^\nu \ell$, $\ell$ odd) and $\lambda_B=\|v_B\|\in \R$ where $(\Ld^*)^B = \Z v_B$ and $\g=BL_b$ is the generator of 
$\Ld\backslash \G$.

\begin{rem}
Expression \eqref{generaletas} for $\eta(s)$ in Theorem \ref{teo.etasz2r} is very simple, compare with the general expression \eqref{etas cyclic}. By \eqref{etaux2}, a priori, all the functions $\zeta(s,\tf{1}{2^t}), \zeta(s,\tf{3}{2^t}), \ldots, \zeta(s,\tf{2^t-1}{2^t})$ should appear in $\eta(s)$, for every $2\le t \le r$. However, as the proof of the theorem shows, great cancellations take place and Hurwitz zeta functions for only one $t$, namely
$\zeta(s,\tf{j}{2^{\nu+2}})$ for odd $j=1,3,5, \ldots, 2^{\nu+2}-1$, contribute to $\eta(s)$. 
\end{rem}

\begin{rem}\label{reti}
Let $M_\G$ be an orientable $\Z_{2^r}$-manifold, $r\ge 2$, of dimension $n=2m+1$, $m$ odd, with $\G= \langle BL_b, L_\Lambda \rangle$ and $n_B=1$. Suppose $B=\mathrm{diag}(B',1)$ with $B'\in \mathrm{SO}(n-1)$ and 
$\Lambda = \Ld' \bigoplus^\perp \Z e_n$
with $\Ld'$ a $B'$-stable lattice in $\R^{n-1}$. In this case, $\ell_\g=\ell$ is odd, i.e.\@ $\nu=0$, for if not $\G$ would have torsion elements other than the identity. This is the case, for instance, for the tetracosm $M_1$ (see \eqref{tetracosmo}), the $\Z_{2^r}$-manifolds in family $\mathcal{F}$ (see Section \ref{sec.family}) and the exceptional $\Z_p$-manifolds, $p$ odd prime (see \cite{GMP}, Proposition 2.2). This is not the case in general, as Example \ref{nonstandard} below shows.
\end{rem}

\begin{rem} \label{rem  35} 
By (the proof of) Theorem \ref{teo.etasz2r} and its previous results, the eta function of an $n$-dimensional $\Z_{2^r}$-manifold 
($n=4h-1$) is non-trivial if and only if $r\ge 2$ and $c_1=1$, $c_2=0$, $c_{2^r}, c_4\ge 1$ with $c_4$ odd. 
Hence, by \eqref{eq. charpol},
$p_{_B}(x) = \Phi_1(x) \Phi_4 (x)^{c_4} \Phi_{2^r}(x)^{c_{2^r}} g(x)$ 
where $g(x) =\prod_{i=3}^{r-1} \Phi_{2^i}(x)^{c_{2^i}}$ with $c_{2^i}\ge 0$, $i=3,\ldots,r-1$. 
\end{rem}

\begin{defi} \label{def nr}
For any fixed $r\ge 2$, let $n_r$ be the \emph{minimal dimension} for a $\Z_{2^r}$-manifold having non-trivial eta function. 
\end{defi}

Thus, $n_2=3$ for $\Z_4$-manifolds and for $r\ge 3$ we have 
\begin{equation}\label{nr}
n_r = \min\limits_{B\in G_{n,r}^*} \{ \deg p_{_B} \} = 2^{r-1}+3, \qquad r\ge 3,
\end{equation}
where $G_{n,r}^* = \{B\in \mathrm{GL}(n) : o(B) = 2^r \text{ and $B$ has no $(-1)$-eigenvalues} \}$, 
corresponding to the decomposition $p_{_B}(x) = \Phi_1(x) \Phi_4 (x) \Phi_{2^r}(x) = (x-1)(x^2+1)(x^{2^{r-1}}+1)$ in the above remark,
that is  $g(x)=1$ (i.e.\@ $c_8=c_{16}=\cdots = c_{2^{r-1}} =0$). 
In Section 5 we will show that $\Z_{2^r}$-manifolds having non-trivial eta function exist for every $r\ge 2$ and in every dimension $n\ge n_r+4k$, $k\ge0$.

\section{Eta invariant and examples}
As a direct consequence of Theorem \ref{teo.etasz2r} we can now obtain the $\eta$-invariant of any $n$-dimensional $\Z_{2^r}$-manifold. It is, up to a sign, a positive power of 2 which does not depend on $n$ nor on $r$, but on the nature of the integral holonomy representation only.

\begin{prop} \label{prop eta0}
Let $M_\G$ be a $\Z_{2^r}$-manifold of $\dim n\equiv 3$ mod 4 with $F= \langle B \rangle$ and put $\ell_\g=2^\nu \ell$ with $0\le \nu \le r-2$ and $\ell$ odd. If $n_B>1$ or $r=1$, then $\eta=0$. If $n_B=1$ and $r\ge 2$, then
\begin{equation}\label{generaleta0}
\eta(M_\G) = \sigma_\g \, 2^{f(B)-2} \in \Z,
\end{equation}
where $\sigma_\g=-\sigma_{v_B} \, (-1)^{[\f{\ell}{2}]} \in \{\pm 1\}$ 
and $f(B)$ is defined in \eqref{fb}. 
\end{prop}
\begin{proof}
By \eqref{generaletas}, 
using that $\zeta(0,a) = \tf12 -a$, we simply get
$$\eta(M_\G) = \eta(0) = \sigma_\g \, 2^{f(B)-1} \big\{ (\tf 12 - \tf 14) - (\tf12-\tf 34) \big\} = \sigma_\g \, 2^{f(B)-2}.$$ 
We have $c_{2^r}, c_4 \ge 1$, by the comments after \eqref{mB}, and hence $f(B)\ge 2$. Thus, $\eta(M_\G) \in \Z$, and the result follows.
\end{proof}

\begin{rem}
(i) Expression \eqref{generaleta0} is very simple (compare with the general formula \eqref{cyclic eta0}). 

(ii) The $\eta$-invariant is, up to sign, determined by $f(B)$. To compute this number, one only needs to know the integral holonomy representation. However, to determine the sign, one needs to know $\g$ and $\Ld$ explicitly. 

(iii) Clearly,  $\eta(s)$ determines the $\eta$-invariant. The converse is not true in general, because of the dependence on $\nu$ and $\lambda_B$ in \eqref{generaleta0}. However, $\eta(s)=0$ if and only if $\eta=0$.

(iv) By \eqref{generaleta0}, the \textit{reduced eta invariant} $\bar \eta = \tf12 (\eta+d_0)$ mod $\Z$ of any $\Z_{2^r}$-manifold is 0 or $\tf12$. Also, $d_0 = \dim \ker \mathcal{D}$ can be computed by using Theorem 3.5 and Proposition~3.7 in \cite{MP.JGA11}. 
\end{rem}

In dimension 3, there are $\Z_{2^r}$-manifolds for $r=1,2$ only. Up to diffeomorphism, there are three $\Z_2$-manifolds and there is only one $\Z_4$-manifold, given by
\begin{equation}\label{tetracosmo}
M_1 = \G \backslash \R^3, \quad \G = \langle \g = BL_{\tf{e_3}{4}}, L_{\Z^3}\rangle, \quad 
B = \mathrm{diag}(J_1,1), \quad J_1=\left( \begin{smallmatrix} & -1 \\ 1 & \end{smallmatrix} \right). 
\end{equation}
This manifold, known as the \emph{tetracosm} after \cite{RC}, gives a nice example in spectral geometry, being one of the `spectral twins', i.e.\@ the only two isospectral-on-functions and non-isometric compact 3-manifolds (\cite{DR}).

\begin{coro} \label{coro eta0}
Let $M_\G$ be a $\Z_{2^r}$-manifold of dimension $n\equiv 3$ mod 4. If $n \ge 7$ then $\eta \in 2\Z$. If $n=3$ then  
$\eta = 0$ for $\Z_2$-manifolds and $\eta = \pm 1$ for $\Z_4$-manifolds. 
\end{coro}
\begin{proof}
We know that $f(B)\ge 2$ and, by \eqref{generaleta0}, $\eta(M_\G) \in 2\Z$ if and only if $f(B)\ge 3$. By Corollary~\ref{coroz2r} we have $c_{2^r}, c_1\ge 1$. By Remark \ref{rem 35}, $\eta\ne 0$ if and only if $r\ge 2$ and $c_1=0$, $c_2=0$, $c_4,c_{2^r}\ge 1$.
Thus, in dimension 7, if $\eta \ne 0$, we have $c_1, c_4, c_{2^r}\ge 1$ with $r>2$, and hence $\eta\in 2\Z$.
On the other hand, $f(B)=2$ if and only if $r=2$ and $c_4=1$ (i.e.\@ $r=2$, $m=1$) and thus $M_\G$ is a $3$-dimensional $\Z_4$-manifold. 
Thus, $M_\G$ is diffeomorphic to the tetracosm $M_1$ defined in \eqref{tetracosmo}. It is known that $\eta(M_1) = -1$ (see \cite[\S 3.1]{MP.AGAG12}). 
Since the $\eta$-invariant is preserved by diffeomorphisms up to sign, and this sign changes with a change of orientation (see \cite{APS1}), the result readily follows.
\end{proof}

Next, we will illustrate with two examples the method for computing $\eta$-invariants given by Proposition \ref{prop eta0}.  
\begin{ejem}[\emph{tetracosm}]
\label{tetracosm}
We now compute the $\eta(s)$ and $\eta$-invariant of $M_1$ (see \eqref{tetracosmo}).

\sk (i)  Note that $p_{_B}(x) = (x-1)(x^2+1)$, hence $f(B)=2$ and 
$\eta(M_1) = -\sigma_B(-1)^{[\ell_\g/2]}$. Also, $v_B=e_3$ and $x_B=x(\tf{\pi}{2})$, then $\sigma_B=1$ and $\ell_\g=1$. 
Thus $\eta(M_1)=-1$ and 
$$\eta(s) = \tf{-2}{(8\pi)^s} \, \{\zeta(s,\tf14)-\zeta(s,\tf34)\}.$$

\sk (ii) Suppose now that we take $\g^3 = B^3L_{\f34 e_3}$ as the generator of $\G$. Since 
$B^3 = \mathrm{diag}(-J_1,1)$, we have $x_{B^3} = x(-\tf{\pi}{2})$, hence $\sigma_{B^3} = -\sigma_B=-1$, and  $\ell_{\g^3}=3$. Thus, 
we also get $\eta(M_1)= -1$. On the other hand, let $M_1' = \G' \backslash \R^3$, with $\G' = \langle \g' = B^3 L_{\f14 e_3}, L_{\Z^3}\rangle$. 
We have $\sigma_{B^3}=-1$ and $\ell_\g = 1$ and thus $\eta(M_1')= 1$. Taking $(\g')^3=(B^3)^3L_{\f34 e_3}=BL_{\f34 e_3}$ as the generator of $\G'$, we also get $\eta(M_1')=1$.

\sk (iii) It is well known that any diffeomorphism between cfm's is given by conjugation of the corresponding Bieberbach groups by an element in the affine group. Suppose that $C(BL_b)C^{-1} = B'L_b$ with $C\in \mathrm{GL}(3)$. Then, $CBC^{-1}L_{Cb}= B^3L_b$, i.e.\@ $CBC^{-1}=B^3$ and $Cb=b$. Then, one can take $C = \mathrm{diag}(1,-1,1)$ and thus
$\varphi_{_C} (M_1) = C \G C^{-1} \backslash \R^3 = \G' \backslash \R^3 = M_1'$.
Since $\det C = -1$, $\varphi_{_C}$ is an orientation reversing diffeomorphism between $M_1$ and $M_1'$, and $M_1' = M_1^-$, the tetracosm with the opposite orientation. We saw in (i) that $\eta(M_1)=-\eta(M_1^-)$.
\end{ejem}

In Section \ref{sec.family} we will define a family of $\Z_{2^r}$-manifolds (with $\eta \ne 0$) having integral holonomy representations of a special kind (see \eqref{Cr}, \eqref{Bj}), that we will refer to as \textit{regular $\Z_{2^r}$-representations}. 
It is difficult, in general, to construct non-regular $\Z_{2^r}$-representations. One way to do that, is to look up at $\Z_{2^r}$-manifolds in the classification of low dimensional cfm's ($\dim \le 7$) in \textsc{carat} \cite{carat}, and assemble some different representations together (taking some care with the translation lattices). However, the resulting associated $\Z_{2^r}$-manifold will have $\eta=0$, in general. 
In the next example we construct a $\Z_8$-manifold having non-regular integral representation with $\eta \ne 0$. 

\begin{ejem}[\textit{$\Z_8$-manifold, integral holonomy representation, $\eta \ne 0$}]
\label{nonstandard}
Consider the matrix $\tilde B = \mathrm{diag}(K,J_1) \in \mathrm{SL}(5,\Z)\times \mathrm{SO}(2,\Z) \subset \mathrm{GL}(7,\Z) $, where
$$K = {\tiny \left( \begin{array}{rrrrr}
              1 & 0 & 0 & 1 & 0 \\
              0 & 0 & 1 & 0 & 0 \\
              0 & 0 & 0 & -1& 0 \\
             -1 & 0 & 0 & 0 &-1 \\
             -1 &-1 & 0 & -1& 0
\end{array} \right)} \qquad \text{and} \qquad J_1 = {\scriptsize \left( \begin{array}{rr} 0&-1\\ 1&0 \end{array} \right)}.$$
It is immediate to check that $K$ has order 8, with eigenvalues $\pm e^{\f{\pi i}{4}}, \pm e^{\f{3\pi i}{4}}, 1$ and that $J_1$ has order 4 with eigenvalues $\pm i$. Take the lattice $\Ld = \Ld_5 \oplus \Ld_2 \subset \R^7$, where $\Ld_5 = \Z f_1 \oplus \cdots \oplus \Z f_5$ is $K$-stable and $\Ld_2 = \Z e_6 \oplus \Z e_7$. Also, one checks that $n_{\tilde B} = n_K=1$ with 
$\Lambda^{\tilde B} = \Lambda^K = \Z (f_1-f_5)$. 
We claim that 
$$\tilde \G =  \langle \tilde \g = \tilde B L_{\f12 f_1}, L_{\Ld}\rangle$$ 
is a discrete cocompact torsion-free subgroup of $\mathrm{Af{}f}(\R^7)$. 
In fact, the lattice $\Ld_5$ exists, since $K$ appears as a subrepresentation of the 6-dimensional Bieberbach group with point group 
$\Z_8$ given by the $\Z$-class labelled 468.1.2 in \cite{carat}. 
Also, by looking at the Bieberbach group with holonomy group $\Z_2 \times \Z_8$ given by the $\Z$-class labeled 4407.1.3, 
one deduces that $\tfrac {1}{2} f_1$ can be used as translation vector for $K$ (and $\tilde B$), hence giving rise to a torsion-free group.

Since $\tilde B$ is conjugate in $\mathrm{GL}(7)$ to
$B = \mathrm{diag}(x(\tf{\pi}{4}),x(\tf{3\pi}{4}),x(\tf{\pi}{2}),1) \in T_7 \subset \mathrm{SO}(7)$ (or to $B' = \mathrm{diag}(x(\tf{-\pi}{4}),x(\tf{3\pi}{4}),x(\tf{\pi}{2}),1)$, hence $\sigma_{v_B} = 1$ or $-1$, respectively), there exists $C\in \mathrm{GL}(7)$ such that $C\tilde B C^{-1}= B$. In this way, we have that
$$   \G = C\tilde \G C^{-1} =  \langle \g =  B L_{\f{1}{2}C f_1}, L_{C\Ld} \rangle \subset \mathrm{I}^+(\R^7) \,,$$
and hence $\G$ is a 7-dimensional Bieberbach group. It is thus clear that $M = \G \backslash \R^7$ is an orientable $\Z_8$-manifold having integral holonomy representation given by the matrix $B$.

Now, one has that $f(B)=3$, since 
$$p_B(x) = p_K(x) p_{J_1}(x) = \big((x-1)(x^4+1)\big) (x^2+1) = \Phi_1(x)\Phi_4(x)\Phi_8(x),$$ 
and thus $\eta(M) = \pm 2$, by \eqref{generaleta0}. The sign $\sigma_\g=\pm 1$ can be determined provided one knows $C$ and $\Ld_5$ explicitly. 
\end{ejem}

\begin{rem}
In \cite{Sz2}, by using results in \cite{Do}, the $\eta$-invariants of 7-dimensional cfm's $M$ having cyclic holonomy group with a special holonomy representation are computed. 
The expression for $\eta(M)$ involves sums of products of cotangents at special angles. 
For such $M$, Theorem 1 claims that $\eta(M)\in \Z$. However, the integrality of $\eta$ comes out after computations with the software 
`Mathematica'. 
There are 126 non-diffeomorphic $\Z_{2^r}$-manifolds in this family (roughly $\tf 13$ of the total), involving only $r=1,2,3$. 
Our Proposition~\ref{prop eta0} assures that indeed $\eta \in \Z$ for these $\Z_{2^r}$-manifolds, and also allows one to compute the $\eta$-invariant in all the cases not covered by the mentioned theorem (see the table before Example 2 in \cite{Sz2}). 
\end{rem}

\section{A distinguished family of $\Z_{2^r}$-manifolds} \label{sec.family}
\subsection{The family $\mathcal{F}$}
For any $r \in \N$, we will construct an infinite family of orientable $\Z_{2^r}$-manifolds in dimensions $n=4h-1$, $h\ge 1$,  
each having holonomy group $F \subset \mathrm{SO}(n,\Z)$. 
For $r\in \N$, let $I_r$ be the $r \times r$ identity matrix and put  
\begin{equation}\label{Cr}
C_r = \left( \begin{smallmatrix} & -1 \\ I_{2^r-1} & \end{smallmatrix} \right)
\qquad \text{and} \qquad
J_r = \left( \begin{smallmatrix} & J_{r-1} \\ I_{2^r} & \end{smallmatrix} \right)
\end{equation}
where $J_0 = (-1)$. Thus, for instance, $C_1 = \left( \begin{smallmatrix} &-1 \\ I_1& \end{smallmatrix} \right) = J_1$, $J_2=\left( \begin{smallmatrix} &J_1 \\ I_2& \end{smallmatrix} \right)$ and {\small $J_3=\left( \begin{smallmatrix} &&J_1 \\ &I_2& \\ I_4&& \end{smallmatrix} \right)$}. It is easy to check that $C_r, J_r \in \mathrm{SO}(2^r,\Z)$.
For instance, for $r\le 3$ we have

{\small 
\begin{center}
\renewcommand{\arraystretch}{1.5}
\begin{tabular}{c|c|c|c|c}
  $r$ & $C_r$ & $J_r$ & order & size  \\ \hline
  1 & $\left( \begin{smallmatrix} & -1 \\ 1 & \end{smallmatrix} \right)$  & $\left( \begin{smallmatrix} & -1 \\ 1 & \end{smallmatrix} \right)$  & 4 & 2 \\ 
	  2 & ${\scriptsize \left( \begin{smallmatrix} &&& -1 \\ 1&&& \\ &1&& \\ &&1& \end{smallmatrix} \right)}$ & ${\scriptsize \left( \begin{smallmatrix} &&& -1 \\ &&1& \\ 1&&& \\ &1&& \end{smallmatrix} \right) }$ 
& 8 & 4 \\ 
  3 & $\tiny{ \left( \begin{smallmatrix} &&&&&&& -1 \\ 1&&&&&&& \\ &1&&&&&& \\ &&1&&&&& \\ &&&1&&&& \\ &&&&1&&& \\ &&&&&1&&  \\ &&&&&&1& \end{smallmatrix} \right) } $ & $\tiny{ \left( \begin{smallmatrix} &&&&&&& -1 \\ &&&&&&1& \\ &&&&1&&& \\ &&&&&1&& \\ 1&&&&&&& \\ &1&&&&&& \\ &&1&&&&&  \\ &&&1&&&& \end{smallmatrix} \right)}$ 
	 & 16 & 8 
\end{tabular}
\end{center}}

Since $C_r$ is the companion matrix of the cyclotomic polynomial $\Phi_{2^{r+1}}(x) = x^{2^{r}}+1$, it has order $o(C_r)=2^{r+1}$ and its eigenvalues are the primitive $2^{r+1}$-th roots of unity. Similarly, one can check that $J_r$ has order $2^{r+1}$ and is conjugate to $C_r$ in $\mathrm{GL}(2^r,\R)$, hence with the same eigenvalues as $C_r$.
Thus, the rotation angles for both $C_r$ and $J_r$ are 
$$\tf{\pi}{2^r}, \tf{3\pi}{2^r}, \tf{5\pi}{2^r}, \ldots, \tf{(2^r-1)\pi}{2^r} \,.$$ 
Since $C_r$ and $J_r$ do not have $\pm 1$-eigenvalues, we have $n_{\pm C_r} = n_{\pm J_r} = 0$, for every $r\ge 2$.

\sk
For any integer $r \ge 2$ and $j_{r-1}, \ldots, j_1 \in \N_0$, with $j_{r-1}>0$, we put $j(r)=(j_{r-1},\ldots,j_2,j_1)$ and define the matrices
\begin{gather}
\begin{aligned} \label{Bj}
B_{{j}(r)} & =
\mathrm{diag} \big( \underbrace{C_{r-1},\ldots, C_{r-1}}_{j_{r-1}}, \ldots, \underbrace{C_1,\ldots, C_1}_{j_1}, 1 \big), \sk \\
B_{{j}(r)}' & = \mathrm{diag} \big( \underbrace{J_{r-1},\ldots, J_{r-1}}_{ j_{r-1}}, \ldots, \underbrace{J_1,\ldots, J_1}_{j_1}, 1 \big).
\end{aligned}
\end{gather}
It is clear that $B_{j(r)}$ and $B_{j(r)}'$ have order $2^{r}$ and belong to $\mathrm{SO}(n,\Z)$, where
\begin{equation}\label{dimension}
n = 2^{r-1} j_{r-1} + \cdots + 4j_2 + 2j_1  + 1 = 2(2^{r-2} j_{r-1} + \cdots + 2j_2 + j_1) + 1 \,.
\end{equation}
Thus, $n=2m+1$ with $m = 2^{r-2} j_{r-1} + \cdots + 2j_2 + j_1$, and $m$ is odd if and only if $j_1$ is odd.

More generally, for each $r, n$ and $j_{r-1}, \ldots,j_1$ satisfying \eqref{dimension} we can take pairs $k_i, k_i' \in \N_0$ such that $k_i+k_i'=j_i$ for $i=1,\ldots, r-1$ and define 
\begin{equation}\label{Bkk'}
B_{k(r)} = \mathrm{diag} \big( \underbrace{C_{r-1},\ldots, C_{r-1}}_{k_{r-1}},  \underbrace{J_{r-1},\ldots, J_{r-1}}_{k'_{r-1}}, \ldots,
\underbrace{C_1,\ldots, C_1}_{k_1}, \underbrace{J_1,\ldots, J_1}_{k_1'}, 1 \big)\,.
\end{equation}
Since $C_1=J_1$, for simplicity we will take $k_1=j_1$ and $k_1'=0$.
In this way, $B_{k(r)} = B_{j(r)}$ if $k(r) = (j_{r-1},0,\ldots,j_2,0,j_1)$ and 
$B_{k(r)} = B_{j(r)}'$ if $k(r) = (0,j_{r-1},\ldots,0,j_2,j_1)$.
Also, $n_{B}=1$ and $x_B$ is as given in \eqref{xB2r0} with $c_{2^i+1}=j_i$, $i=1,\ldots,r-1$, that is
\begin{equation}\label{xjr}
x_B = x_{j_{r-1},\ldots,j_1} = x(\underbrace{ \tf{\pi}{2^{r-1}}, \tf{3\pi}{2^{r-1}}, \ldots,  \tf{(2^{r-1}-1)\pi}{2^{r-1}} }_{j_{r-1}}, 
\ldots, \underbrace{ \tf{\pi}{8}, \tf{3\pi}{8},  \tf{5\pi}{8}, \tf{7\pi}{8} }_{j_3}, \underbrace{ \tf{\pi}{4}, \tf{3\pi}{4} }_{j_2}, \underbrace{ \tf{\pi}{2} }_{j_1}).
\end{equation}

Define the Bieberbach groups
\begin{equation}\label{Z2r group}
    \Gamma_{k(r)} := \langle \g= B_{k(r)} L_{b_r}, L_{\Ld} \rangle, \qquad b_r = \tf{1}{2^r} e_n
\end{equation} 
where $j_1$ is odd and $\Ld = \Z e_1 \oplus \cdots \oplus \Z e_n$ is the canonical lattice in $\R^n$. Since $B_{k(r)} \in \mathrm{SO}(n)$ and $F = \langle B_{k(r)} \rangle \simeq \Z_{2^r}$, we have the associated orientable $\Z_{2^r}$-manifold
\begin{equation}\label{Z2r man}
    M_{k(r)} := \Gamma_{k(r)} \backslash \R^n 
\end{equation}
of dimension $n=2m+1 \equiv 3 \mod 4$.

\begin{defi} \label{defi f}
For a fixed $r$, let $\mathcal{F}_{r}$ denote the set of all $\Z_{2^r}$-manifolds as in \eqref{Z2r man} and let
$\mathcal{F} = \bigcup_{r =1}^\infty \mathcal{F}_{r}.$ 
Also, put $\mathcal{F}(n) = \{M \in \mathcal{F} : \dim M = n \}$ and $\mathcal{F}_r(n)
= \mathcal{F}_r\cap \mathcal{F}(n)$.
\end{defi} 
Then, we have that $\mathcal{F}(n)$ consists of $\Z_{2^r}$-manifolds with $1 \le r \le t= \lceil \log_2 n \rceil $. In other words, if $2^{t-1} < n \le 2^{t}-1$,
$$\mathcal{F}(n) = \mathcal{F}_1(n) \cup \mathcal{F}_2(n) \cup \cdots \cup \mathcal{F}_t(n).$$ 

\sk 
Moreover, for any fixed $r$, the number of $\Z_{2^r}$-manifolds in dimension $n$ of the form $M_{j(r)}$ equals the number of partitions 
of $n$ into the first $r$-powers of two, i.e., $2^t$ with $0\le t \le r-1 $. This number is known as the \textsl{binary partition function} and is denoted by $b(n)$. Since $j_1 \ge 1$, we have that $\#\mathcal{F}_r(n) \ge b(n-3)$. Mahler showed that the logarithm of $b(n)$ grows like $(\log n)^2/2\log 2$ as $n$ grows to infinity (\cite{Ma}, see also \cite{Pe}).
Thus, asymptotically, we have 
$$\#\mathcal{F}_r(n) \sim (n-3)^{\sqrt{\log_2(n-3)}}.$$   

On the other hand note that, for a given $r$, the minimal dimension for $M\in \mathcal{F}_r$ is $n=2^{r-1}+3$ for $r\ge 3$ 
(since $j_{r-1},j_1\ge1$ and $j_i=0$ for $2\le i\le r-2$) and $n=3$ for $r=2$ (since $j_{r-1}=j_1$ in this case). Thus, 
$\min \{\dim M : M\in \mathcal{F}_r\} = n_r$ (see \eqref{nr}).

\begin{ejem} \label{tetra}
We now describe the manifolds in $\mathcal{F}$ in the lowest dimensions. There is only one $3$-manifold in $\mathcal{F}$, 
the tetracosm $M_1$ in \eqref{tetracosmo}. 
In dimension $7$, there are 3 manifolds in $\mathcal{F}$; two $\Z_8$-manifolds $M_{1,0,1}, M_{0,1,1}$, determined by 
$\mathrm{diag}(J_2,J_1,1)$, $\mathrm{diag}(C_2,J_1,1)$, and the $\Z_4$-manifold $M_3$ given by 
$\mathrm{diag}(J_1,J_1,J_1,1)$. In dimension $11$, we have 8 manifolds in $\mathcal{F}$; two $\Z_{16}$-manifolds $M_{1,0,0,0,1}$ and  $M_{0,1,0,0,1}$, given by the matrices $\mathrm{diag}(C_3,J_1,1)$ and $\mathrm{diag}(J_3,J_1,1)$; five $\Z_8$-manifolds 
$M_{2,0,1}$, $M_{1,1,1}$, $M_{0,2,1}$, $M_{1,0,3}$ and $M_{0,1,3}$ given respectively by $\mathrm{diag}(C_2,C_2,J_1,1)$, 
$\mathrm{diag}(C_2,J_2,J_1,1)$, $\mathrm{diag}(J_2,J_2,J_1,1)$, $\mathrm{diag}(C_2,J_1,J_1,J_1,1)$ and $\mathrm{diag}(J_2,J_1,J_1,J_1,1)$; and the $\Z_4$-manifold $M_5$ determined by $\mathrm{diag}(J_1,J_1,J_1,J_1,J_1,1)$.
\end{ejem}

\subsection{Homology and $\eta$-invariants}
We will now compute the first integral homology and cohomology groups for all the manifolds in $\mathcal{F}$, showing that $H_1(M,\Z)$ has only $2$-torsion.
\begin{prop} \label{teoH1}
Let $M=M_{k(r)} \in \mathcal{F}$. Then
\begin{equation}\label{H1}
H_1(M,\Z) \simeq \Z \oplus \Z_2^{j_{r-1} + \cdots + j_1} 
\end{equation}
with $j_i=k_i+k_i'$ for $1\le i \le r-1$, and
\begin{equation}\label{H^1}
H^1(M,\Z) \simeq \Z, \qquad H^1(M,\Z_2) \simeq \Z_2^{j_{r-1}+\cdots+j_1+1}.
\end{equation}
\end{prop}
\begin{proof}
We will first compute $H_1(M_\G,\Z)= \G / [\G,\G]$, where $\G=\G_{k(r)}=\langle \g, L_{e_1}, \ldots, L_{e_n}\rangle$. There are 
3 kinds of commutators: 
$[L_\ld,L_{\ld'}]=I$, $[\g,L_\ld]$ and $[\g,\g']$. 
Since $F$ is cyclic, every element in $\G$ is of the form $\g^i L_\ld$ and thus 
$[\g^i L_\ld,\g^j L_\ld'] = [\g^i, \g^j][L_\ld,L_{\ld'}] = I$ for every $i,j \in \Z$, $\ld,\ld'\in \Ld$. Also, 
$[\g,L_\ld] = BL_b L_\ld L_{-b} B^{-1} L_{-\ld} = L_{B\ld-\ld}$.
Therefore, 
$$[\G,\G] = \langle [\g,L_\ld]  : \ld \in \Ld \rangle = L_{(B-I)\Ld} \,.$$

We now study the action of each of the blocks $C_i$ and $J_i$ on $\Ld_i \subseteq \Ld$, with $\Ld_i \simeq \Z^{2^i}$. For every 
$1\le i \le r-1$, let $e_1, \ldots, e_{2^i}$ be any $\Z$-basis of $\Ld_i$. For $C_i$, we have 
$$(C_i - I_{2^i})e_j  = e_{j+1} - e_j, \quad 1\le j\le 2^i-1, \qquad \text{and} \qquad (C_i - I_{2^i})e_{2^i} =-e_1 - e_{2^i}.$$ 
By putting $f_j = e_{j+1} -e_j$ for $1\le j\le 2^i-1$ and $f_{2^i} = -e_1 - e_{2^i}$, we have that 
$$(f_1 + \cdots + f_{j-1}) - (f_j + \cdots + f_{2^i}) = 2e_j, \qquad 1\le j\le 2^i,$$
and thus $2e_j \in L_{(C_i-I_{2^i})\Ld_i}$ for every $1\le j \le 2^i$. Furthermore, since 
$L_{e_{j}}-L_{e_{j-1}} \in [\G,\G]$ for $1\le j\le 2^i$, we have that 
$L_{e_j} \sim L_{e_k}$, $1\le j,k\le 2^i$, in the quotient $\G/[\G,\G]$. Therefore, 
$$\langle L_{e_1},\ldots,L_{e_{2^i}}\rangle / \langle L_{f_1},\ldots,L_{f_{2^i}} \rangle \simeq \Z_2,$$ 
and thus, every block $C_i$ gives rise to one $\Z_2$ in the quotient $\G / L_{(B-I)\Ld}$.

On the other hand, for the block $J_i$, it will be convenient to look up at the $e_j$'s with $j$ in the intervals 
$1 \le j \le 2^{i-1}$, $2^{i-1}+1 \le j \le 2^{i-1}+2^{i-2}$, $2^{i-1}+2^{i-2}+1 \le j \le  2^{i-1}+2^{i-2}+2^{i-3}$, and so on. 
We then have 
$$(J_i - I_{2^i})e_{j(l,k)} = e_{2^{i-l}+k} - e_{j(l,k)}, \qquad 1\le k\le 2^i, \: 1\le l\le r,$$ 
where $j(l,k) = 2^{i-1}+2^{i-2}+\cdots +2^{i-(l-1)} + k$.
As before, by putting $g_j=e_{2^{i-l}+k} - e_{j(l,k)}$ and looking at the sums of the form $\sum_{j} \pm  g_j$ one sees that 
$2e_j \in L_{(J_i-I_{2^i})\Ld_i}$ for every $1\le j \le 2^i$. 
Also, we again have that $L_{e_{j}} \sim L_{e_{k}} \in [\G,\G]$ for $1\le j,k \le 2^i$. Thus, every block $J_i$ induces a $\Z_2$ in the quotient $\G / L_{(B-I)\Ld}$.

Moreover, since $\g^{2^r}=L_{e_n}$, it is clear that $\langle \g, L_{e_n}\rangle$ generates an infinite cyclic group in $\G / L_{(B-I)\Ld}$. As a result of all these observations we get expression \eqref{H1}.

\sk
Now, it is known that $H^n(M,\Z) \simeq F_n \oplus T_{n-1}$, where $F_n$ and $T_n$ are the free and the torsion part of $H_n(M,\Z)$, respectively. Since $H_0(M,\Z) \simeq \Z$, by \eqref{H1} we get $H^1(M,\Z) = \Z$, as desired. Finally, by the universal coefficient theorem, we have
\begin{eqnarray*}
H^1(M,\Z_2)  &\simeq&  \mathrm{Hom}(H_1(M,\Z),\Z_2) \oplus \mathrm{Ext}(H_0(M,\Z),\Z_2) \\
&\simeq& \mathrm{Hom}(\Z\oplus \Z^{j_{r-1}+\cdots+j_2+j_1}),
\Z_2),
\end{eqnarray*}
where we have used that $\mathrm{Ext}(H_0(M,\Z),\Z_2) \simeq \mathrm{Ext}(\Z,\Z_2) \simeq 0$, since $\Z$ is projective. Finally, since 
$\mathrm{Hom}(\oplus_i G_i,G) = \bigoplus_i \mathrm{Hom}(G_i,G)$ and $\mathrm{Hom}(\Z,\Z_2) = \mathrm{Hom}(\Z_2,\Z_2) =
\Z_2$ we obtain that $H^1(M,\Z_2)\simeq \Z_2^{j_{r-1}+\cdots+j_2+j_1+1}$, and the proof is now complete.
\end{proof}

As a result, we can count the number of spin structures of the manifolds in $\mathcal{F}$.
\begin{coro}\label{spins}
Every $M=M_{k(r)} \in \mathcal{F}$ is spin and has $2^{j_{r-1}+\cdots+j_1+1}$ spin structures. 
\end{coro}
\begin{proof}
By applying the methods used in \cite{MP.MZ04} or \cite{Po.UMA05} (see Theorem 2.1 or Proposition 2.2, respectively, and their previous comments), one can prove that $M$ is a spin manifold (and get all the spin structures explicitly). The number of spin structures of $M$ is then given by $\#H^1(M,\Z_2) = 2^{j_{r-1}+\cdots+j_1+1}$ (\cite{LM}).  
\end{proof}

\begin{quest} \label{q1}
Let $a(r)=(a_{r-1},a_{r-1}',\ldots,a_2,a_2',j_1)$ and $b(r)=(b_{r-1},b_{r-1}',\ldots,b_2,b_2',j_1)$ be two $(2r-1)$-tuples satisfying 
$a_i+a_i' = b_i+b_i' = j_i$ for $2\le i \le r-1$, but with $a(r)\ne b(r)$. Then, $M_{a(r)}$ and $M_{b(r)}$ have, in general, different integral representations. Since they have the same eigenvalues, $C_r$ and $J_r$ are conjugate in $\R$. Are they still conjugate in $\Z$? In other words, are $M_{a(r)}$ and $M_{b(r)}$ equivalent as cfm's? Proposition \ref{teoH1} gives no answer to this question.
\end{quest}

\begin{rem}
In \cite{Po.UMA05}, (2.1)--(2.3), we have defined a family $\mathcal{F}^n = \{M_{j,k,l}\}$ of $n$-dimensional $\Z_4$-manifolds, $n=2m+1$, and we obtained that $H_1(M_{j,k,l},\Z) \simeq \Z^l \oplus \Z_2^{j+k}$ (see Lemma~2.1). If $m$ is odd and $\Lambda =\Z^n$, the manifolds $M_{j,0,1} \in \mathcal{F}^n$ are exactly the $\Z_4$-manifolds of the form $M_{j(r)}=M_j \in \mathcal{F}_2$ in this paper. This is in agreement with \eqref{H1}.
Corollary \ref{spins} says, for example, that the tetracosm $M_1$ has $2^{1+1}=4$ spin structures. This is also in coincidence with Proposition 2.2 and Corollary 2.3 in \cite{Po.UMA05}.
 \end{rem}

\begin{rem}
It is possible to define a bigger family $\tilde{\mathcal{F}}$, using also matrices of order $\le 2$. That is, we can add $2 \times 2$ blocks
$J = (\begin{smallmatrix} & 1 \\ 1 & \end{smallmatrix})$ and $\pm I = (\begin{smallmatrix} \pm 1& \\ & \pm 1 \end{smallmatrix})$.  
Thus, consider the matrix
$$B_{\kappa(r)} := \mathrm{diag} \big( B'_{k(r)}, \underbrace{J,\ldots, J}_{k_0},  \underbrace{-I,\ldots, -I}_{k_0'}, 
\underbrace{I,\ldots,I}_{i_0}, 1 \big)$$
with $k_0,k_0'$ even and where $B'_{k(r)}$ denotes the matrix $B_{k(r)}$ in \eqref{Bkk'} with the 1 in the last position removed. Similarly as in \eqref{Z2r group}--\eqref{Z2r man} we define the group $\Gamma_{\kappa(r)}$ and the corresponding $\Z_{2^r}$-manifold 
$M_{\kappa(r)}$ of dimension 
$n=2^{r-1}j_{r-1} + \cdots + 4j_2 + 2(j_1 + j_0 + i_0) +1$, where $j_0 = k_0+k_0'$. Hence, $j_1 + i_0$ must be odd for $n$ to be congruent to $3$ mod $4$. For $M_{\kappa(r)} \in \tilde{\mathcal{F}}$, proceeding similarly as in the proof of Proposition \ref{teoH1}, one obtains
$$H_1(M_{\kappa(r)},\Z) = \Z_2^{j_{r-1} + \cdots + j_1 + k_0'} \oplus \Z^{k_0 + i_0+1}.$$
Similar results can be obtained for $H^1(M_{\kappa(r)},\Z)$ and $H^1(M_{\kappa(r)},\Z_2)$. However, we have that 
$\eta(M_{\kappa(r)})=0$ unless $M_{\kappa(r)} \in \mathcal{F}$, i.e.\@ $j_0=i_0=0$ and $M_{\kappa(r)} = M_{k(r)}$. 
\end{rem}

As a consequence of our previous results we have that, for manifolds in $\mathcal{F}$, the $\eta$-invariant has a strong topological meaning. In fact, it is related with the order of the torsion group $T$ (which equals the $2$-torsion) of the first integral homology group and with the number of spin structures.  
\begin{prop} \label{prop.etasz2r}
If $M=M_{k(r)} \in \mathcal{F}$ then 
\begin{equation} \label{etaH1}
\eta (M) = -2^{j_{r-1}+\cdots +j_1-1} = - \tf12 |T| = \tf14 \# \mathrm{Spin}(M) \ne 0, 
\end{equation}
where $T$ is the torsion subgroup of $H_1(M,\Z)$ and $\mathrm{Spin}(M)=\{\text{spin structures on $M$}\}$.
\end{prop}
\begin{proof}
Suppose $M=M_{k(r)}$ of dimension $n=2m+1$, with $m$ odd, and $k_i+k_i'=j_i$ for $1\le i \le r$.
Then, $m=2^{r-2}j_{r-1} + \cdots +2j_2+j_1$ with $j_i = c_{2^{i+1}}$ for $1\le i \le r-1$.
Since $v_B=e_n$ we have $\ell_\g=1$, for $\nu=0$ and $\ell=1$, and thus the expression for $\eta(s)$ follows from \eqref{generaletas}.
Since $\sigma_B=1$ and $f(B) = j_{r-1}+\cdots +j_1+1$, by \eqref{generaleta0}, we have $\eta=-2^{j_{r-1}+\cdots + j_1-1}$. The result readily follows from \eqref{H1} in Proposition \ref{teoH1}. 
\end{proof}

\begin{quest}\label{q2}
For the $\Z_8$-manifold $M$ of Example \ref{nonstandard}, Proposition \ref{teoH1} and Proposition \ref{prop.etasz2r} do not apply. Proceeding similarly as in (the proof of) Proposition \ref{teoH1}, we can check that $H_1(M,\Z) = \Z \oplus \Z_2^{2}$ and, 
since $\eta(M)=\pm 2$, we still have $\eta(M)=\pm \tf12 |T|$ as in \eqref{etaH1}. 
Does this phenomenon hold in general or is there a $\Z_{2^r}$-manifold $M\notin \mathcal{F}$ with non-trivial $\eta$-invariant 
such that $\eta(M) \ne -\tf12 |T|$? 
\end{quest}

\begin{quest}\label{q3}
If $M$ is an arbitrary $\Z_{2^r}$-manifold, is $\eta(M)$ completely determined by $H_1(M,\Z)$ as it is the case for manifolds in $\mathcal{F}$? The best we can say is the statement of Lemma \ref{lemin} below. 
\end{quest}

\begin{rem}\label{eta ell}
By definition, every $\Z_{2^r}$-manifold $M \in \mathcal{F}$ has a generator $\g = B L_b$ with $b=\tf{1}{2^r} e_n$ and hence $\ell_\g=1$. For a given $M$, we can define $M_{\ell}$ with $b$ replaced by $b_\ell=\tf{\ell}{2^r} e_n$, $\ell$ odd (these manifolds are diffeomorphic but non-isometric to each other). Clearly, we have $\ell_\g=\ell$ and
$\eta (M_\ell) = (-1)^{[\ell/2]} \, \eta (M)$, for every $\ell$ odd.
\end{rem}

Note that, for fixed $r$, the $\eta$-invariant determines the eta function $\eta(s)$. 
In fact, by Theorem \ref{teo.etasz2r} and Proposition \ref{prop eta0}, we have
\begin{equation} \label{etas F}
\eta (s,M) = -\tf{2\eta(M)}{(2^{r-1-\nu}\pi \lambda_B)^s} L(s,\chi_4). 
\end{equation}
Since manifolds in $\mathcal{F}$ have $\nu=0$ ($\ell_\g=1$) and $\lambda_B=1$, it is clear that if 
$M, M' \in \mathcal{F}_r$ then 
$$\eta(M)=\eta(M') \quad \Rightarrow \quad \eta(s,M)=\eta(s,M').$$ 
This may not be the case in general because of the numbers $r$, $\nu$ and $\lambda_B$; 
since it could well happen that $2^{r-1-\nu}\lambda_B = 2^{r'-1-\nu'}\lambda_B'$.

\section{The image of $\eta^*$} \label{sec eta*}
We now consider the following subfamilies of the set $\mathcal{M}$ of all $\Z_{2^r}$-manifolds: 
the subset $\mathcal{M}_r$ of $\Z_{2^r}$-manifolds with fixed $r$ (arbitrary dimension), the subset $\mathcal{M}(n)$ of 
$n$-dimensional $\Z_{2^r}$-manifolds (arbitrary $r$) and the subset $\mathcal{M}_r(n) = \mathcal{M}_r \cap \mathcal{M}(n)$, with both $r$ and $n$ fixed. Clearly, $\mathcal{F} \subset \mathcal{M}$, $\mathcal{F}_r \subset \mathcal{M}_r$, $\mathcal{F}(n) \subset \mathcal{M}(n)$ and 
$\mathcal{F}_r(n) \subset \mathcal{M}_r(n)$.

The aim of this section is to study the image of the map
\begin{equation} \label{eta*}
\eta^*  \, : \,  \mathcal{M} \rightarrow \Z, \qquad M \mapsto \eta(M)
\end{equation}
and its restrictions $\eta_r^*$ to $\mathcal{M}_r$ and  $\eta_{(n)}^*$ to $\mathcal{M}(n)$. 
By Proposition \ref{prop eta0}, it is clear that 
$$\mathrm{Im} \, \eta^* \subseteq \{0,\pm 1\} \cup \{\pm 2^k : k\in \N\}.$$
The extreme cases, i.e.\@ $\Z_2$-manifolds and $3$-dimensional $\Z_{2^r}$-manifolds, are trivial ones, since 
$\eta=0$ if $r=1$ and $\eta = 0, \pm 1$ if $n=3$; i.e.\@ $\mathrm{Im}\, \eta_1^* = \{0\}$ and $\mathrm{Im}\, \eta_{(3)}^* = \{0,\pm 1\}$.

Since we are considering dimensions $n\equiv 3$ mod $4$, for $M\in \mathcal{F}$, $n$ is as given in \eqref{dimension} with $j_1$ odd. Out of all possible partitions of $n$ into powers of $2$, the $2$-adic expansion of $n$ is a proper partition with the minimum number of parts. Let $\tau(n)$ be this number, i.e.\@ 
\begin{equation} \label{tau}
\tau(n) = \tau \qquad \text{if} \qquad n=2^{a_\tau} + \cdots + 2^{a_2} + 2^{a_1}, \quad 0\le a_1 < a_2 < \cdots < a_\tau.
\end{equation}
Equivalently, if $n= b_m 2^m + \cdots + b_3 2^3 + b_2 2^2 +b_1 2+b_0$ for some $m$, with $b_i\in\{0,1\}$, $1\le i \le m$,
then $\tau(n)$ equals the number of non zero $b_i$'s. 
In our case, we have $\tau(3)=2$ and $\tau(n)\ge 3$ for $n\ge 7$. 
Furthermore, $\tau(n_r)=3$ for any $r\ge 2$ and $\tau(n)\ge 4$ for $n\ne n_r$, where we recall that, by \eqref{nr} and 
\eqref{dimension},  
\begin{equation} \label{nr2}
n_r = \min_{M\in \mathcal{F}_r} \{ \dim M \} = \min_{M\in \mathcal{M}_r} \{\dim M: \eta(M)\ne 0\} = 2^{r-1}+3.
\end{equation}

We now show that given a general $\Z_{2^r}$-manifold $M$ with $\eta\ne 0$, the value $\eta(M)$ can be obtained as the eta invariant of some manifold in the family $\mathcal{F}$.
\begin{lema} \label{lemin}
If $M\in \mathcal{M}$ then $\eta(M)=0$ or $\eta(M) = \eta(M_{k(r)})$ for some $M_{k(r)} \in \mathcal{F}$. 
\end{lema}
\begin{proof}
Let $M$ be any orientable $\Z_{2^r}$-manifold of dimension $n$, with point group 
$F=\langle B\rangle$. 
By Proposition \ref{prop eta0}, $\eta(M)=0$ if $n_B>1$ or $r=1$, while $\eta(M) = \pm 2^{f(B)-2}$ otherwise. 
So, assume that $n_B=1$ and $r\ge 2$. 
Since $B$ is conjugate to $x_B=x_{j_{r-1},\ldots,j_1}$ as in \eqref{xjr} (or to $x_B'$, see Proposition \ref{lema angles}), $B$ is also conjugate in $\mathrm{GL}(n)$ to the matrix $C_B = \mathrm{diag}(B_{j(r)},1)$ with $k(r)=(j_{r-1},\ldots,j_1)$ as in \eqref{Bj}. Thus, $\eta(M)$ is determined (up to sign) by the number of matrices $C_i$ or $J_i$, $1\le i \le r-1$, in $C_B$. Finally, if $\eta(M) = -\eta(M_{k(r)})$, then considering $M_{k(r)}^-$ (the manifold $M_{k(r)}$ with the opposite orientation) we get $\eta(M) = \eta(M_{k(r)}^-)$. 
\end{proof}

We now describe the image of $\eta_{(n)}^*$ for every dimension $n$.

\begin{teo}\label{im eta} 
Let $n=2m+1$, with $m$ odd, be fixed. Then, $\mathrm{Im}\, \eta_{(n)}^* = \{0\}$  for $r=1$, $\mathrm{Im}\, \eta_{(n)}^* = \{0,\pm 2^{m-1} \}$ for $r=2$ and 
\begin{equation}\label{Imn*}
\mathrm{Im}\, \eta_{(n)}^* = \{0,\pm 2^{\tau(n)-2}, \pm 2^{\tau(n)-1}, \ldots, \pm 2^{m-1}\}, \qquad r\ge 3.
\end{equation}
In particular,  
$\mathrm{Im}\, \eta_{(n_r)}^* = \{ 0,\pm 2, \pm 2^2,\ldots, \pm 2^{2^{r-2}} = \pm  2^{\f{n_r-3}{2}} \}$.
Therefore, $\mathrm{Im}\, \eta_1^* = \{ 0\}$, $\mathrm{Im}\, \eta_2^* = \{0,\pm 4^k\}_{k\in \N_0}$ 
and $\mathrm{Im}\, \eta_r^* = \{0,\pm 2^k\}_{k\in \N}$ for $r\ge 3$.
\end{teo}
\begin{proof}
By Lemma \ref{lemin}, we can assume that $M \in \mathcal{F}_r(n)$ and, without loss of generality, that $M=M_ {j(r)}$ as in \eqref{Bj}--\eqref{Z2r man}. By Proposition~\ref{prop eta0}, $\operatorname{Im} \eta_{(3)}^* =\{0,\pm 1\}$, with the values $\pm 1$ attained by the tetracosm. From now on, suppose that $n\ge 7$ and $\eta\ne 0$. Hence, by \eqref{generaleta0}, $n_B=1$ and $r\ge 2$. 
Fix the dimension
$$n= 2(2^{r-2} j_{r-1} + \cdots + 2j_2 + j_1) +1 \qquad j_{r-1}>0, \quad \text{$j_1$ odd.}$$ 

If $r=2$, then $n=2 j_1+1=2m+1$ and thus $f(B) = m+1$. Hence, $\eta(M) = \pm 2^{m-1}$ by \eqref{generaleta0}. The positive value $2^{m-1}$ is attained by the manifold $M_m \in \mathcal{F}$ determined by $B_m = \mathrm{diag}(C_1,\ldots,C_1,1)$, with $C_1$ repeated $m$-times.

\sk 
Let $r\ge 3$. If $n=n_r=2^{r-1}+3$, all the values $2, 4, 8, \ldots, 2^{m-1}$ are attained by the $\eta$-invariants 
on the following manifolds. Note that $M^{(1)}=M_{1,0,\ldots,0,1} \in \mathcal{F}_r$, determined by the matrix $\mathrm{diag}(C_{r-1},C_1,1)$, has 3 blocks and hence $\eta(M^{(1)})=2$. By replacing the block $C_{r-1}$ by 2 blocks $C_{r-2}$ we get 
$M^{(2)} = M_{0,2,0,\ldots,0,1}\in \mathcal{F}_{r-1}$ determined by  $\mathrm{diag}(C_{r-2}, C_{r-2}, C_1, 1)$ with 4 blocks, and hence 
$\eta(M^{(2)})=2^{4-2}=4$. Now, by replacing one block $C_{r-2}$ from the previous matrix by 2 blocks $C_{r-3}$ we get $M^{(3)} = M_{0,1,2,0,\ldots,0,1} \in \mathcal{F}_{r-1}$ determined by  $\mathrm{diag}(C_{r-2}, C_{r-3}, C_{r-3}, C_1, 1)$ with 5 blocks, and hence $\eta(M^{(2)})=2^{5-3}=8$. It is clear that, by repeating this \textsl{`splitting block'} procedure, i.e.\@ by replacing some block $C_{j_i}$ of $M^{(i)}$ by 2 blocks $C_{j_i-1}$ (this changes the $\eta$-invariant keeping the dimension unaltered), we get a finite sequence of $\Z_{2^t}$-manifolds 
$$M^{(1)}, M^{(2)}, \ldots, M^{(m-1)}$$
(with different $t$'s, $3\le t\le r$), determined by diagonal block matrices with $3,4,\ldots, m+1$ blocks respectively, and hence, with corresponding $\eta$-invariants 
$2^1,2^2,\ldots, 2^{m-1}$.
Since $n=2m+1$, we have 
$m=2^{r-2}+1$ and hence $m-1=2^{r-2}=\tf{n_3-3}{2}$.
Here, $M^{(m-1)}$ is the $\Z_4$-manifold determined by the matrix  $\mathrm{diag}(C_1, \ldots, C_1,1)$, with $C_1$ repeated $m$ times.
The splitting of blocks is not unique, but there is at least one. Also, by dimension issues, it is clear that $2^{m-1}$ is the maximum possible value for $\eta$.

In case $n \ne n_r$, we proceed similarly as before. We begin with $M^{(1)}=M_{j_{r-1},\ldots,j_1}$, where $j_{r-1}=j_1=1$ and 
$j_{i_1},\ldots,j_{i_{\tau-3}} \ge 1$ with $\tau=\tau(n)$  --the other $j_k$'s being $0$--, such that 
$$n=2^{r-1} + (\sum\limits_{k=1}^{\tau-3} 2^{i_k}j_{i_k}) + 3 = n_r + \sum\limits_{k=1}^{\tau-3} 2^{i_k}j_{i_k}.$$  
By replacing some block $C_{j_i}$ by 2 blocks $C_{j_i-1}$ and iterating this process, we get a sequence 
$M^{(1)}, M^{(2)},\ldots, M^{(m-\tau+2)}$ of $\Z_{2^t}$-manifolds (with different $t$'s, $t\le r$), respectively determined by diagonal block matrices with $\tau,\tau+1,\ldots, m+1$ blocks, and hence with corresponding $\eta$-invariants 
$2^{\tau-2},2^{\tau-3},\ldots, 2^{m-1}$.

\sk 
To get trivial $\eta$-invariants for any $n$ and $r$, just take the previous manifolds and replace one $J_1$ by 
$\pm I =(\begin{smallmatrix} \pm 1 &0 \\ 0& \pm 1 \end{smallmatrix})$ (these manifolds will be not in $\mathcal{F}$). To get the negative values, since $\eta(M^-) = -\eta(M)$, we just change the orientation of every $M$ previously used.

The remaining assertions in the statement follow directly from the previous ones. 
\end{proof}

As a direct result of the theorem, by taking $n$ and $r$ big enough we get every possible power of $2$ as the $\eta$-invariant of 
a $\Z_{2^r}$-manifold. That is, 
$$\eta(\{\text{$\Z_{2^r}$-manifolds} : r\in \N\}) = {0} \cup \{ \pm2^k : k\in \N_0 \}.$$

\begin{ejem}\label{ejemplin}
We illustrate the results in Theorem \ref{im eta}. Using \eqref{Imn*}, in Table~\ref{tabla2} below, we give the values of 
$\eta(M)$ for $M$ a $\Z_{2^r}$-manifold in the lowest dimensions. For each fixed value of $n$, we give the $2$-adic expansion, $\tau(n)$, the highest possible $r$ for a $\Z_{2^r}$-manifold in this dimension, and all the allowed values of $\eta$ (only non-negative values for simplicity). 
{\small 
\renewcommand{\arraystretch}{1}
\begin{table}[h]
\caption{$\eta$-invariants for $\Z_{2^r}$-manifolds, $3\le n \le 63$.}  
\label{tabla2}
\begin{center}
\begin{tabular}{l|c|c|l} 
dimension $n$         & $\tau(n)$ & $\max r$       & $\eta$-invariant \\ 
\Xhline{2\arrayrulewidth}
$3=2+1=n_2$                &    2      & 2 & $0, 1$              \\ \hline
$7=4+2+1=n_3$            &    3      & 3 & $0, 2,2^2$           \\ \hline 
$11=8+2+1=n_4$          &    3      & 4 & $0, 2, 2^2, 2^3, 2^4$    \\
$15=8+4+2+1$              &    4      & 4 & $0, 2^2, 2^3, 2^4, 2^5, 2^6$ \\\hline 
$19=16+2+1=n_5$         &    3      & 5 & $0, 2, 2^2,\ldots, 2^8$ \\
$23=16+4+2+1$            &    4      & 5 & $0, 2^2, 2^3,\ldots, 2^{10}$  \\
$27=16+8+2+1$            &    4      & 5 & $0, 2^2, 2^3,\ldots,  2^{12}$ \\
$31=16+8+4+2+1$        &    5      & 5 & $0, 2^3, 2^4,\ldots, 2^{14}$ \\ \hline
$35=32+2+1=n_6$         &    3      & 6 & $0, 2, 2^2,\ldots, 2^{16}$ \\
$39=32+4+2+1$            &   4      & 6 & $0, 2^2, 2^3, \ldots,  2^{18}$ \\
$43=32+8+2+1$           &    4      & 6 & $0, 2^2, 2^3, \ldots,  2^{20}$ \\
$47=32+8+4+2+1$        &    5      & 6 & $0, 2^3, 2^4, \ldots, 2^{22}$ \\
$51=32+16+2+1$          &    4      & 6 & $0, 2^2, 2^3, \ldots,  2^{24}$ \\
$55=32+16+4+2+1$          &   5      & 6 & $0, 2^3, 2^4, \ldots, 2^{26}$ \\
$59=32+16+8+2+1$          &    5      & 6 & $0, 2^3, 2^4, \ldots,  2^{28}$ \\
$63=32+16+8+4+2+1$          &    6     & 6 & $0, 2^4, 2^5, \ldots,  2^{30}$ 
\end{tabular}
\end{center}
\end{table}}
\end{ejem}

Our next goal is to show that there exists an infinite number of infinite families of $\Z_{2^r}$-manifolds in $\mathcal{F}$ 
having constant $\eta$-invariant, one for every possible positive power of two. 

Let $n_{r,k}$ denote the minimal dimension for a $\Z_{2^r}$-manifold with $\eta=2^k$, i.e.
\begin{equation} \label{nrk}
n_{r,k} = \min_{M\in \mathcal{M}_r} \{\dim M \,:\, \eta(M)=2^k\}.
\end{equation}
By Theorem \ref{im eta}, this number is well defined for $r\ge 3$, $k\ge1$ and for $r=2$, $k=1$ or $k$ even. Note that $n_{r,1}=n_r$ and for $r\ge 2, k\ge 1$, 
by \eqref{generaleta0},  we have 
\begin{align*}
n_{r,1} & \le n_{r,2}  \le  n_{r,3} \le \cdots  \le  n_{r,k}  \le  n_{r,k+1} \le \cdots \\ 
n_{2,k} & \le  n_{3,k}  \le   n_{4,k} \le \cdots \le  n_{t,k}  \le  n_{t+1,k} \le \cdots 
\end{align*} 
We now compute these dimensions.
\begin{lema}\label{lem nrk}
We have $n_{2,2k} = 3+2k$ for $k\ge 0$ and  for any $r\ge 3$, $k\ge 1$
$$n_{r,k}= n_r+[\tf k2]= 2^{r-1}+[\tf k2]+3.$$
\end{lema}
\begin{proof}
Let $M\in \mathcal{M}_r$ with $r\ge 2$. By Lemma \ref{lemin} we can assume that $M=M_{j(r)}\in \mathcal{F}$. 
If $r=2$, it is clear that $n_{2,2k}=2k+3$, attained by the manifold $M_{2k+1}$ determined by the matrix 
$B_{2k+1} = \mathrm{diag}(J_1, \ldots, J_1, 1)$, with $J_1$ repeated $2k+1$ times.

Now consider the case $r\ge 3$. It is clear that $n_{r,1}=n_r=2^{r-1}+3$, attained by the manifold $M_{1,0,\ldots,0,1}$ determined by the matrix $B_{1,0,\ldots, 0,1}=\mathrm{diag}(J_{r-1},J_1,1)$. To get $\eta=2^2$ we need 4 blocks, so the minimum dimension where this can be achieved is $7+4=11$ given by the manifold $M_{1,0,\ldots,0,1,1}$ determined by the matrix $B_{1,0,\ldots,0,1,1}=\mathrm{diag}(J_{r-1},J_2,J_1,1)$. Similarly, for $\eta=2^3$, the minimum dimension is also $11$, attained by the manifold $M_{1,0,\ldots,0,3}$ determined by the matrix $B_{1,0,\ldots,0,3}=\mathrm{diag}(J_{r-1},J_1,J_1,J_1,1)$ (just split one block $J_2$ into two $J_1$'s). In general, the minimal dimension needed to get $\eta=2^{2k}$ or $2^{2k+1}$ is the same. In fact,
$n_{r,2k}= 2^{r-1} + 4+2(2k-2)+3$ and $n_{r,2k+1}= 2^{r-1} + 2((2k+1)-1)+3$, that is
$$n_{r,2k}=n_{r,2k+1}=2^{r-1}+4k+3=n_r+4k,$$ 
from which the expression in the statement directly follows.
These dimensions are respectively attained by the manifolds $M_{2,2k-1}$ and $M_{1,2k+1}$, determined by the matrices 
$B_{1,0,\ldots,0,1,2k-1}$ $=\mathrm{diag}(J_{r-1},J_2,J_1,\ldots,J_1,1)$, $J_1$ repeated $2k-1$ times, and $B_{1,0,\ldots,0,2k+1}=$ 
$\mathrm{diag}(J_{r-1},J_1,\ldots,J_1,1)$, $J_1$ repeated $2k+1$ times. 
\end{proof}

\begin{ejem} (i) Let us construct $\Z_{2^7}$-manifolds $M$, $M'$ with $\eta(M)=2^{11}$, $\eta(M')=2^{10}$. Since $[\tf{11}2]=5$, the minimal dimension for $M$ is given by $n_{7,11}= n_7+4\cdot 5= 2^6+3+20= 87$. We need 13 `blocks', so take $M=M_{1,0,0,0,0,11}$ given by $B=\mathrm{diag}(E_6,J_1,\ldots,J_1,1)$ with $J_1$ repeated $11$ times, where $E_6=C_6$ or $J_6$. 
As before, $n_{7,10}=87$. Now, we need 12 blocks, thus we take $M'=M_{1,1,0,0,0,10}$ given by 
$B=\mathrm{diag}(E_6,E_2, J_1,\ldots,J_1,1)$, $J_1$ repeated 10 times, with $E_i=C_i$ or $J_i$ for $i=2,6$.

\sk (ii) 
In Table \ref{tabla3} below, we give all $\Z_{2^r}$-manifolds in family $\mathcal{F}$ having $\eta=2^3$ in dimensions 
$n=4k+3 \le 35$. These manifolds are defined by diagonal block matrices $B_{j(r)}$ with the blocks in the set 
$\{J_4, J_3, J_2, J_1, J_0\}$, as defined in \eqref{Bj}.
\begin{table}[h]
\caption{All $\Z_{2^r}$-manifolds in $\mathcal{F}$ with $\eta=8$ in $\dim n \le 35$.}  
\label{tabla3}
\begin{center}
\begin{tabular}{c|c|ccccc|c|c}
$\dim$ & partition of $n$ & $J_4$ & $J_3$ & $J_2$ & $J_1$ & $J_0$ & $r$ & $F$ \\ 
\hline 
 11 & 4+2+2+2+1 & 0 & 0 & 1 & 3 & 1 & 3 & $\Z_8$ \\ 
 15 & 4+4+4+2+1 & 0 & 0 & 3 & 1 & 1 & 3 & $\Z_8$ \\
 15 & 8+2+2+2+1 & 0 & 1 & 0 & 3 & 1 & 4 & $\Z_{16}$\\ 
 19 & 8+4+4+2+1 & 0 & 1 & 2 & 1 & 1 & 4 & $\Z_{16}$\\ 
 23 & 8+8+4+2+1 & 0 & 2 & 1 & 1 & 1 & 4 & $\Z_{16}$\\  
 23 & 16+2+2+2+1& 1 & 0 & 0 & 3 & 1 & 5 & $\Z_{32}$\\ 
 27 & 8+8+8+2+1 & 0 & 3 & 0 & 1 & 1 & 4 & $\Z_{16}$\\
 27 & 16+4+4+2+1& 1 & 0 & 2 & 1 & 1 & 5 & $\Z_{32}$\\ 
 31 & 16+8+4+2+1& 1 & 1 & 1 & 1 & 1 & 5 & $\Z_{32}$
\end{tabular}
\end{center}
\end{table}
\end{ejem}

We now show that there are infinite families of  $\Z_{2^r}$-manifolds with prescribed constant $\eta$-invariant and growing dimensions. 
\begin{coro} \label{coro =eta}
For every positive integer $k$ there is a family $\mathcal{G}_k=\{M_i\}_{i=1}^\infty \subset \mathcal{F}$ 
such that $\eta(M_i)= 2^k$ for every $i$ and 
$\dim M_i \nearrow \infty$. In particular, for every $r\ge 2$, one can take each $M_r \in \mathcal{G}_k$ having holonomy group of order $2^r$ and minimal dimension $n_{r,k}$.  
\end{coro}
\begin{proof}
It is sufficient to prove the second claim in the statement. Let $k\ge 1$, $r\ge 2$. 
By \eqref{generaleta0}, to get $\eta=2^k$ , we need to take a manifold in $\mathcal{F}$ with exactly $k+2$ blocks. 
Consider the matrix 
$$B_r = \mathrm{diag}(J_{r-1},\underbrace{J_2,\ldots,J_2}_{(k-1)-\mathrm{times}},J_1,1).$$ 
The induced manifold $M_r$ in $\mathcal{F}$ has holonomy group of order $2^r$, $\eta=2^k$ and dimension 
$\dim M_r = 2^{r-1}+4(k-1)+1 =2^{r-1}+4k-3\ge n_{r,k}$.
To obtain minimal dimensions, just proceed as in the proof of Lemma \ref{lem nrk}, taking $B_r = \mathrm{diag}(J_{r-1},J_2,J_1,\ldots,J_1,1)$ or $B_r = \mathrm{diag}(J_{r-1},J_1,\ldots,J_1,1)$, depending on whether $k$ is even or odd.
\end{proof}

\section{An alternative expression for $\eta$} \label{sec doneli}
In \cite{MP.JGA11}, we gave an expression for the $\eta$-invariant of any cfm (i.e., arbitrary holonomy representations, any translation lattice) following a direct approach; i.e., we first computed the multiplicities of the eigenvalues, then we found $\eta(s)$ and finally we obtained $\eta$ by evaluation at $s=0$ (see Theorems 3.3, 3.5 and 4.2 in \cite{MP.JGA11}).

On the other hand, Donnelly has previously obtained an expression for the $\eta$-invariant for more general manifolds $M$ (compact oriented Riemannian manifolds with a Lie group acting by isometries on it) in an indirect way (\cite{Do}, Theorem 3.4), by first computing the signature of a manifold $\tilde M$ with $M=\partial \tilde M$, and then using the Atiyah-Patodi-Singer index theorem for manifolds with boundary in \cite{APS1}.
For a very special kind of cfm's, namely those having holonomy group $F \subset \textrm{SO}(n,\Z)$ and canonical lattice $\Lambda=\Z^n$, the expression for $\eta$ drastically simplifies (\cite{Do}, Proposition 4.12).

Since both methods are quite different, i.e.\@ representation theoretical vs.\@ topological ones, it is the author feeling that it is interesting to compare both expressions for the $\eta$-invariant, when possible (i.e.\@ in the special case considered in \cite{Do}). 

\sk 
Donnelly considered cfm's having holonomy group $F\subset \mathrm{SO}(n,\Z)$ where each $B\in F$ is of the form $B=\mathrm{diag}(B',1)$, $B' \in \mathrm{SO}(n-1,\Z)$. This is a rather restricted family, since if $B\in \mathrm{SO}(n,\Z)$ then necessarily
\begin{equation}\label{bsonz}
B e_i = \pm e_{j(B,i)} \qquad 1 \le i \le n-1, \qquad Be_n=e_n,
\end{equation}
where $\{e_i\}$ is the canonical basis of $\R^n$. Such $B$ induces a permutation matrix $P_B$ given by $P_B e_i = e_{j(B,i)}$. Then, we have the decomposition
$P_B = P_{B,1} \cdots P_{B,c}$ 
into disjoint cycles. In other words, $c=c(B)$  
is the number of orbits of the action of $B$ on the basis vectors. Clearly, $c(B)=c(B')+1\ge 2$. Notice that the matrices $B_{k(r)}$ in \eqref{Bkk'} satisfy condition \eqref{bsonz} and that 
\begin{equation}\label{cBfB}
c(B_{k(r)}) = j_{r-1} + \cdots + j_1 + 1 = f(B_{k(r)}).
\end{equation}

\sk 
We now recast Donnelly's expression for the $\eta$-invariant in our present notations (with $\alpha=1$ the trivial representation) and with our sign conventions ($\sigma_\gamma$ as in Theorem \ref{teo.etasz2r}).
\begin{prop}[\cite{Do}, Proposition 4.12] \label{teo.donnelly}
Let $M_\Gamma$ be a compact flat manifold of dimension $n \equiv 3$ mod~$4$ with translation lattice $\Lambda = \Z^n$ 
such that every $\g=BL_b \in \G$ is of the form  $B=\operatorname{diag}(B',1)$ where $B'\in \mathrm{SO}(n-1,\Z)$ and $b=ae_n$, with $a\in \R$. Then,
\begin{equation}\label{doneli}
     \eta^{\mathrm{[Do]}}(M_\G) = \tf{\sigma_{\gamma}}{|F|}  \sum_{BL_b\in (\Lambda \backslash \G)'} 2^{c(B')} \, 
     \Big( \prod_{j=1}^{m} \cot \big( \tf{t_j(x_B)}{2} \big) \Big) \, \cot( \pi a).
\end{equation}
\end{prop}

Note that \eqref{doneli} is formally very similar to \eqref{cyclic eta0}, which is valid for arbitrary cfm's. 
Both expressions involve sums of trigonometric products at special angles. The main difference seems to be the factors $2^{c(B)}$ for $BL_b\in \G$.

\sk 
Since $\Z_{2^r}$-manifolds in $\mathcal{F}$ satisfy the conditions in Proposition \ref{teo.donnelly}, they are specially suited for comparison, since Donnelly's formula applies in this case. For $\Z_{2^r}$-manifolds in $\mathcal{F}$, the expression \eqref{doneli} can be explicitly computed. 
\begin{prop} \label{etaz2rdoneli}
If $M=M_{k(r)} \in \mathcal{F}$ then $\eta^{\mathrm{[Do]}}(M) = 2^{j_{r-1} + \cdots + j_1-1}$. 
\end{prop}
\begin{proof}
Starting from \eqref{doneli}, and using that $\sigma_B=1$, $\ell_\g=1$, $F = \langle B \rangle$ is cyclic of order $N=2^r$, 
with $B=B_{k(r)}$, and that $B^k\in F_1'$ if and only if $k$ is odd,  we get
\begin{equation}\label{etadon1}
\eta^{\mathrm{[Do]}}(M_\G) = \ 2^{c(B')-r} \, \sum_{k\in I_{2^r}^*} \, 
\Big( \prod_{j=1}^{m} \cot \big( \tf{k t_j}{2} \big) \Big) \, \cot( \tf{\pi k}{2^r})
\end{equation}
since $c(B^k)=c(B)$ for any $k$ odd. 
Now, by \eqref{xjr}, we have
\begin{eqnarray*}
\prod_{j=1}^{m} \cot \big( \tf{k t_j}{2} \big)
& = &
\Bigg( \prod_{k\in I_{2^r}^*}
\cot (\tf{jk\pi}{2^{r+1}}) \Bigg)^{j_{r-1}} \cdots \Bigg( \prod_{k\in I_{8}^*} 
\cot (\tf{jk\pi}{16}) \Bigg)^{j_3} \Bigg( \prod_{k\in I_{4}^*} 
\cot (\tf{jk\pi}{8}) \Bigg)^{j_{2}} \cot(\tf{\pi k}{4})^{j_1}
\\ & = & 
(-1)^{[\tf k2]j_1} \, \prod_{i=2}^{r-1} \Bigg( \prod_{k\in I_{2^i}^*} 
\cot (\tf{jk\pi}{2^{i+1}}) \Bigg)^{j_{i}} = (-1)^{[\tf k2]}
\end{eqnarray*}
where in the last equality we have used that $j_1$ is odd and Proposition \ref{lema. cots} in the Appendix.
Hence, by \eqref{etadon1},
$$\eta^{\mathrm{[Do]}}(M_\G) = 2^{c(B')-r} \, \sum_{k\in I_{2^r}^*}  (-1)^{[\tf k2]} \, \cot( \tf{\pi k}{2^r}).$$
Thus, by Proposition \ref{prop. sumcots}, 
$\eta^{\mathrm{[Do]}}(M_\G) =  2^{c(B')-r} \, 2^{r-1} = 2^{j_{r-1}+\cdots + j_1 -1}$, 
as asserted.
\end{proof}

\begin{rem}
(i) By Propositions \ref{prop.etasz2r}, \ref{etaz2rdoneli} and Lemma \ref{lemin}, Proposition \ref{teo.donnelly} actually holds for arbitrary $\Z_{2^r}$-manifolds; that is, for any lattice and any integral holonomy representation.

(ii) By Proposition \ref{prop.etasz2r} and (i), $\eta^{\mathrm{[Do]}}(M_\Gamma) = \eta^{\mathrm{[MP]}}(M_\Gamma)$; i.e.\@ Donnelly's general expression for the $\eta$-invariant, restricted to $\Z_{2^r}$-manifolds, coincides with our expression. 
\end{rem}

Donnelly's expression \eqref{doneli} works for any cfm whose integral holonomy representation is restricted to $\mathrm{SO}(n,\Z)$.
In light of the previous results, one may ask if Proposition \ref{teo.donnelly} can be generalized to hold with more generality. 
For $\Z_{2^r}$-manifolds we have the previous remark.
However, the condition on the holonomy representation cannot be removed for $F \not \simeq \Z_{2^r}$, 
as the next example shows.
\begin{ejem} 
Up to diffeomorphism, there is only one $\Z_3$-manifold in dim $3$, the \emph{triscosm} $M_3 = \Gamma_{3} \backslash \R^3$ where $\G_3 = \langle BL_{\f{e_3}{3}}, L_{\Ld} \rangle$, $B = \textrm{diag}(B',1)$, $B' = ( \begin{smallmatrix} 0&-1 \\ 1&-1 \end{smallmatrix})$, and 
$\Ld = \Ld_{_\mathcal{H}} \oplus \Z e_3$, with $\Ld_{_\mathcal{H}} = \Z e_1 \oplus \Z f_2$, $f_2= -\f12 e_1 + \f{\sqrt 3}2 e_2$, the plane hexagonal lattice. It is known that $\eta(M_3)=-\tf 23$ (\cite{LR}, also \cite{MP.AGAG12}, \cite{Sz2}). However, applying \eqref{doneli}, since $x_{B^k}=x(\tf{2k\pi}3)$ and $c(B^k)=2$, $k=1,2$, one gets 
$$\eta^{[\mathrm{Do}]}(M_3) = -\tf 23 \sum_{k=1,2} \cot(\tf{k\pi}{3})^2 = -\tf 23 (\tf 13+\tf 13) = -\tf 49.$$
Here, $B' \notin \mathrm{SO}(2)$ and, after conjugation, we get $\big( \begin{smallmatrix} \cos(2\pi/3)&-\sin(2\pi/3) \\ 
\sin(2\pi/3)& \cos(2\pi/3) \end{smallmatrix} \big) \in \textrm{SO}(2) \smallsetminus \mathrm{SO}(2,\Z)$. 
\end{ejem}

The $\Z_3$-manifold in the example is a particular case of $\Z_p$-manifolds, $p$ odd prime (classified by Charlap \cite{Ch}). See \cite{GMP}, \cite{MP.TAMS06}, \cite{MP.PAMQ09}, \cite{MP.JGA11} and \cite{MP.AGAG12} for details on the construction/classification of $\Z_p$-manifolds, topological properties and computations of $\eta(s)$ and $\eta$ for both the spin Dirac and the APS operators. 

\begin{rem} \label{zpmanifolds}
Expression \eqref{doneli} cannot be applied to $\Z_p$-manifolds, with $p$ an odd prime, because $\Z_p$-manifolds having non-trivial $\eta$-invariant have integral holonomy representations taking values in $\mathrm{GL}(n,\Z) \smallsetminus \mathrm{SO}(n)$ 
(see \S2 in \cite{MP.PAMQ09}, \S4 in \cite{MP.AGAG12}). As before, one can show (with a little bit more effort) that by applying expression \eqref{doneli} anyway, we get a different result for $\eta$ that when using formula \eqref{cyclic eta0}. 
\end{rem}

\section{Appendix: Trigonometric identities}
The rotation angles of order $2^r$ matrices give rise to interesting non-trivial trigonometric identities involving products and alternating sums of sines and cotangents. 
The following notations will be useful in the sequel. For any positive integer $N$ define 
$$I_N=\{i \in \N: 1 \le i\le N\} \qquad \text{and} \qquad I_N^* = \{j\in I_N : j \,\mathrm{odd}\}.$$
We will need the following technical result.
\begin{lema} \label{lema sns}
Let $k,N\in \N$. If $k$ is odd and $N$ is even then 
\begin{equation} \label{sn*}
S_N^*(k) := \sum_{j\in I_{N-1}^*} [\tf{jk}{N}] = \tf{(k-1)N}{4} \in \Z,
\end{equation} 
where $[\,\cdot\,]$ is the floor function. In particular, $S_N^*(k) \in 2\Z$ for $N=2^r$, $r\ge 2$. 
\end{lema}
\begin{proof}
Let $S_N(k) := \sum_{j=1}^{N-1} [\tf{jk}{N}]$ and note that $S_N(k) = \sum_{j=1}^{N/2} [\tf{jk}{N}] + \sum_{j=1}^{N/2} \big[ \tf{(N-j)k}{N} \big] - [\tf N2 \tf kN]$. Thus, we have $\big[ \tf{(N-j)k}{N} \big] = [k-\tf{jk}{N}] = k+[-\tf{jk}{N}] = k-1-[\tf{jk}{N}]$, since $[x+n]=[x]+n$ for $n\in \Z$ and $[-x]=-1-[x]$ for 
$x\not \in \Z$. Hence,
$$S_N(k) = \sum_{1\le j \le N/2} [\tf{jk}{N}]+ \Big( \tf{(k-1)N}2 - \sum_{1\le j \le N/2} [\tf{jk}{N}] \Big) + [\tf k2] = 
\tf{(k-1)N}2 + \tf{k-1}2 = \tf{(k-1)(N-1)}2.$$
Now, $S_N(k) = S_N^*(k) + S_{N/2}^*(k) + \cdots + S_4^*(k) + S_2^*(k) = S_N^*(k) + S_{N/2}(k)$, 
by induction, and hence,
$S_N^*(k) = S_N(k) - S_{N/2}(k) = \tf{(k-1)}{2} \tf N2$, and we are done.
\end{proof}

\subsubsection*{Sines}
We now compute some products and alternating sums of sines at certain integer multiples of $\frac{\pi}{2^r}$.
\begin{prop} \label{prop. sines2}
If $r,k$ are positive integers with $r\ge 2$ and $k$ odd then 
\begin{equation} \label{sines2}
\prod_{j\in I_{2^r}^*} \sin \big( \tf{jk\pi}{2^r} \big) = \tf{1}{2^{2^{r-1}-1}}.
\end{equation}
\end{prop}
\begin{proof}
Let $N=2^{r}$ and assume $k$ is odd. We will first prove that the product in \eqref{sines2} does not depend on $k$. 
For any fixed odd $k$, there are unique integers $q_j, r_j$ such that $jk=q_j N + r_j$ with $0\le r_j \le N-1$. If $j$ is odd, then $r_j$ must be odd, and hence $r_j\ge 1$. Thus,
\begin{equation}\label{signsin}
\sin(\tf{jk\pi}{N}) = \sin(q_j\pi + \tf{r_j \pi}{N}) = (-1)^{q_j} \sin(\tf{r_j \pi}{N}).
\end{equation}
By modularity, for any $k$ odd we have $\{ 1, 3,\ldots, N-1 \} = \{ k, 3k, \ldots, k(N-1)\}$ mod $N$ and hence 
$\{ 1, 3,\ldots, N-1 \} = \{r_1,r_3,\ldots,r_{N-1}\}$ mod $N$. Therefore,  by \eqref{signsin}
\begin{equation}\label{prodsins}
\prod_{j\in I_{2^r}^*} \sin(\tf{jk\pi}{N}) = 
(-1)^{S_N^*(k)}  \prod_{j\in I_{2^r}^*} \sin(\tf{r_j\pi}{N}) = 
\prod_{j\in I_{2^r}^*} \sin(\tf{j\pi}{N}),
\end{equation}
where we have used that $q_j=[\tf{jk}{N}]$ and $S_N^*(k)$ as in \eqref{sn*} is even by Lemma \ref{lema sns}, since $r\ge 2$.

\sk It only remains to compute the last product in \eqref{prodsins}. 
We will first prove the identity
\begin{equation} \label{sines}
\prod_{j\in I_{2^{r-1}}^*} \sin \big( \tf{j \pi }{2^r} \big) = \tf{\sqrt 2}{2^{2^{r-2}}}. 
\end{equation}
Assume that $r\ge 2$, the case $r=1$ being trivial. Let us define 
$\Pi_{2^r} := \prod_{j=1}^{2^{r-1}} \sin(\tf{j\pi}{2^r})$  and $\Pi_{2^r}^* := \prod_{j\in I_{2^{r-1}}^*} \sin(\tf{j\pi}{2^r})$. 
We recall the identity $\prod_{j=1}^{[\f d2]} \sin(\tf{j\pi}d) = \sqrt{d} / 2^{\f {d-1}2}$,
with $d\in \N$ (see \cite{MP.PAMQ09}, Lemma~3.1, for a proof). By taking $d=2^r$ we have that
$\Pi_{2^r} = 2^{\f r2 - \f{2^r-1}2}$. By induction, we have 
$\Pi_{2^r} = \Pi_{2^r}^* \Pi_{2^{r-1}}^* \cdots \Pi_{2^2}^* \Pi_{2^1}^*  = \Pi_{2^r}^*  \Pi_{2^{r-1}}$
for $r\ge 2$, and thus 
$$\Pi_{2^r}^* = \f{\Pi_{2^r}}{\Pi_{2^{r-1}}} = \f{2^{\f r2 - \f{2^r-1}2}}{2^{\f{r-1}2 - \f{2^{r-1}-1}2}} = \f{\sqrt 2}{2^{2^{r-2}}},$$
as desired. 

By symmetry, $\sin(\tf{(N-j)\pi}{N}) = \sin(\tf{j\pi}{N})$, for $0<j<\f N2$. By using this and \eqref{sines} we obtain
$$\prod_{j\in I_{N}^*} \sin(\tf{j\pi}{N}) = \prod_{j\in I_{2^{r-1}}^*} \big( \sin(\tf{j\pi}{2^{r}}) \big)^2 = 
(\Pi_{2^{r}}^*)^2 = \big( \tf{\sqrt 2}{2^{2^{r-2}}} \big)^2 = \tf{1}{2^{2^{r-1}-1}},$$
and thus the proposition follows. 
\end{proof}

Now, for $r\in \N$ and $t,\omega \in \Z$ we define the sums 
\begin{equation}\label{S,r,t}
\mathcal{S}_{r,t}(\omega) := \sum_{k\in I_{2^r}^*} \, (-1)^{[\tf k2]} \; \sin(\tf{k \omega \pi}{2^t}) \,.
\end{equation}

\begin{prop}\label{prop. sumsines}
Let $r, \nu \in \N$ and $t,\ell \in \Z$ with $\ell$ odd. If $t\le r$ then 
\begin{equation}\label{sumsines}
\mathcal{S}_{r,t}(2^\nu \ell) = \left\{ \begin{array}{cl} 
(-1)^{[\tf{\ell}{2}]} \, 2^{r-1} & \qquad \text{if } t=\nu +1, \msk \\ 
0 & \qquad \text{if } t \ne \nu+1. \end{array} \right. 
\end{equation}
\end{prop}
\begin{proof}
For $t\le 0$ the result is trivial, thus assume $t\ge 1$. Suppose first that $\nu=0$. 
Consider $t=1$. Since for any $k$ odd, $\sin(\tf{k \pi}{2}) = (-1)^{[\tf k2]}$ and 
$\sin(\tf{k \ell \pi}{2}) = (-1)^{[\tf{k\ell}{2}]} = (-1)^{[\tf k2]} (-1)^{[\tf{\ell}{2}]}$, with $\ell$ odd, then we have
$$\mathcal{S}_{r,1}(\ell) = \sum_{k\in I_{2^r}^*} (-1)^{[\tf k2]} \, \sin(\tf{k \ell \pi}{2}) =  
\sum_{k\in I_{2^r}^*} (-1)^{[\tf{\ell}{2}]} = (-1)^{[\tf{\ell}{2}]}  \, 2^{r-1}.$$

Now let $t>1$. 
If $t=r$, we have the angles $\tf{\pi}{2^r}, \tf{3\pi}{2^r},\ldots,\tf{(2^r-1) \pi}{2^r}$ in $[0,\pi]$, and we compare the contributions of $\tf{k\ell \pi}{2^r}$ and $\tf{(2^r-k) \ell \pi }{2^r}$. Since for any positive odd integer $k$ we have 
$\sin(\tf{(2^r-k)\ell\pi}{2^r}) = (-1)^{\ell+1} \sin(\tf{k\ell \pi}{2^r})$
and $(-1)^{[\f{2^r-k}{2}]} = (-1)^{[2^{r-1} - \f k2]} = (-1)^{[-\tf{k}{2}]} = (-1)^{[\tf k2]+1}$, 
we get
$$(-1)^{[\f{2^r-k}{2}]} \sin(\tf{(2^r-k)\pi}{2^r}) = (-1)^\ell (-1)^{[\f{k}{2}]} \sin(\tf{k \ell \pi}{2^r})$$
and thus, since $\ell$ is odd, the contributions of the angles $\tf{k\pi}{2^r}$ and $\tf{(2^r-k)\pi}{2^r}$ in \eqref{sumsines} cancel each other out.
If $t<r$, there are more angles to consider. However, by modularity, it is enough to consider the angles 
$$\tf{\pi}{2^t}, \tf{3\pi}{2^t},\ldots,\tf{(2^t-1) \pi}{2^t},\tf{(2^t+1)\pi }{2^t},\tf{(2^t+3)\pi }{2^t},\ldots,\tf{(2^{t+1}-1) \pi }{2^t}$$ 
in the interval $(0,2\pi)$. In this case, we compare $\tf{k\ell\pi}{2^t}$ with $\tf{(2^t+k)\ell\pi}{2^t}$, 
for any $k=1,3,5,\ldots,2^t-1$. Since $\sin(\theta + \pi) = -\sin(\theta)$  and $(-1)^{[\f{2^t+k}{2}]} = (-1)^{[\f k2]}$, we see again that the contributions of $\tf{k\ell\pi}{2^t}$ and $\tf{(2^t+k)\ell\pi}{2^t}$ cancel out.
In this way, $\mathcal{S}_{r,t}(\ell)=0$ for $t>1$.

Now, consider the case $\nu \ge 1$. Clearly, $\mathcal{S}_{r,t}(2^\nu \ell)=0$ for $\nu\ge t$. For $1\le \nu \le t-1$, note that 
$\mathcal{S}_{r,t}(2\ell)=\mathcal{S}_{r,t-1}(\ell)$. Hence, by induction,
$$\mathcal{S}_{r,t}(2^\nu \ell)=\mathcal{S}_{r,t-\nu}(\ell)=\delta_{t,\nu+1} \, (-1)^{[\tf{\ell}{2}]}  \, 2^{r-1},$$ 
where $\delta$ is the Kronecker function, and thus \eqref{sumsines} holds. 
\end{proof}

\subsubsection*{Cotangents}
We now compute some products and alternating sums of cotangents at some integer multiples of $\frac{\pi}{2^r}$.

\begin{prop} \label{lema. cots}
For any $r, k \in \N$, with $r\ge 2$ and $k$ odd, the following identity holds
\begin{equation}
\label{prodcots}
\prod_{j\in I_{2^{r-1}}^*} \cot(\tf{jk\pi}{2^r}) = 1.
\end{equation}
\end{prop}
\begin{proof}
By using $\sin 2\theta = 2 \sin \theta \cos \theta$ and \eqref{sines2}, for any $r\ge 2$, we get
\begin{equation}\label{cos}
\prod_{j\in I_{2^{r}}^*}
\cos(\tf{j\pi}{2^r}) =
\f{ \prod_{j\in I_{2^{r}}^*} 
\sin(\tf{2j\pi}{2^r})} {2^{2^{r-2}} \prod_{j\in I_{2^{r}}^*} 
\sin(\tf{j\pi}{2^r})} = \f{\tf{1}{2^{2^{r-2}-1}}}{2^{2^{r-2}} \tf{\sqrt 2}{2^{2^{r-2}}}} = \f{\sqrt 2}{2^{2^{r-2}}}.
\end{equation}
Now, the identities \eqref{signsin} and \eqref{prodsins} also hold for $\cos(\f{jk\pi}{N})$ changing every sine by the corresponding cosine. Therefore, the product in \eqref{prodcots} does not depend on $k$. Thus, we get
$\prod_{j\in I_{2^r}^*} \cot(\tf{jk  \pi}{N}) = \prod_{j\in I_{2^r}^*} \cot(\tf{j \pi}{N})=1$,
by \eqref{sines} and \eqref{cos}.
\end{proof}

\begin{prop} \label{prop. sumcots}
Let $r, \ell \in \N$ with $\ell$ odd. If $N$ is either $2^r$ or $2^{r-1}$, with $r\ge 2$, then
\begin{equation}
 \label{sumcots}
    \sum_{j\in I_{N}^*} (-1)^{[\tf k2]} \, \cot(\tf{k \ell \pi}{2^r}) = (-1)^{[\tf{\ell}{2}]} \, \tf N2  \,.
\end{equation}
\end{prop}
\begin{proof}
We will denote by $\Sigma_{2^r}$ and $\Sigma_{2^r}'$ the sums in \eqref{sumcots} corresponding to $N=2^r$ and $2^{r-1}$, respectively.
Note that, $\cot \big( \tf{(2^r-k)\ell \pi}{2^r} \big) = -\cot \big( \f{k \ell \pi}{2^r} \big)$, and hence
$$\Sigma_{2^r} = \cot(\tf{\ell \pi}{2^r}) - \cot(\tf{3 \ell \pi}{2^r}) + \cdots + 
\cot(\tf{(2^r-3) \ell \pi}{2^r}) - \cot(\tf{(2^r -1)\ell\pi}{2^r}) = 2 \, \Sigma_{2^r}' \,.$$
Thus, it is enough to prove that $\Sigma_{2^r}' = (-1)^{[\tf{\ell}{2}]}2^{r-2}$, which we will do by induction on $r$. It is immediate to check that $\cot(\tf{\ell\pi}{4}) = (-1)^{[\f{\ell}2]}$ for any odd $\ell$ and hence, $\Sigma_{2^2}'=(-1)^{[\f{\ell}2]} 2$, and the first step in the induction holds. For the general step, we have
$$\Sigma_{2^r}' = \sum_{k\in I_{2^{r-1}}^*}
(-1)^{[\tf k2]} \cot(\tf{k \ell \pi}{2^r}) = \sum_{k\in I_{2^{r-2}}^*}
(-1)^{[\tf k2]} \; \Big(\! \underbrace{\cot(\tf{k \ell \pi}{2^r}) - \cot\big( \tf{(2^{r-1}- k) \ell \pi}{2^r} \big)}_{C_{k,r}} \Big) \,.$$
Since $\cos \big( \tf{(2^{r-1}-k)\ell\pi}{2^r} \big) = (-1)^{[\f{\ell}2]} \sin(\tf{k\ell\pi}{2^r})$ and $\sin \big( \tf{(2^{r-1}-k)\ell\pi}{2^r} \big) = (-1)^{[\f{\ell}2]} \cos(\tf{k\ell\pi}{2^r})$, by using 
$\cos 2\theta = \cos^2 \theta - \sin^2 \theta$ and $\sin 2\theta = 2\sin \theta \cos \theta$, we have
$$C_{k,r} = \f{\cos(\tf{k\ell\pi}{2^r})}{\sin(\tf{k\ell\pi}{2^r})} - \f{\sin(\tf{k\ell\pi}{2^r})}{\cos(\tf{k\ell\pi}{2^r})}
= \f{\cos^2(\tf{k\ell\pi}{2^r}) - \sin^2(\tf{k\ell\pi}{2^r})}{\sin(\tf{k\ell\pi}{2^r}) \cos(\tf{k\ell\pi}{2^r})}
= \f{2 \cos(\tf{k\ell\pi}{2^{r-1}})}{\sin(\tf{k\ell\pi}{2^{r-1}})} = 2 \cot(\tf{k\ell\pi}{2^{r-1}}).$$

In this way, by the inductive hypothesis we get
$$\Sigma_{2^r}' = 2 \sum_{k\in I_{2^{r-2}}^*}(-1)^{[\tf k2]} \cot(\tf{k \ell \pi}{2^{r-1}})
= 2\, \Sigma_{2^{r-1}}' =  (-1)^{[\f{\ell}2]} 2^{r-2}.$$
and the result thus follows. 
\end{proof}

\subsection*{Acknowledgments}
I am very grateful to Professor Roberto Miatello for useful conversations which led this paper to the final form. 
In particular, I am in debt with him for pointing out to me the subfamily of manifolds in $\mathcal{F}$ determined by the matrices $B'$ in \eqref{Bj}.

\end{document}